\documentclass[10pt]{article}

\usepackage{amsmath}
\usepackage{amsthm}
\usepackage{amssymb}
\usepackage{enumitem}
\usepackage{multicol}
\usepackage{booktabs} 
\usepackage[active]{srcltx}
\usepackage{algorithm}
\usepackage{algpseudocode}
\usepackage{array}
\usepackage[dvipsnames,svgnames]{xcolor}
\usepackage{tikz}
\usetikzlibrary{positioning,arrows.meta}
\usepackage{graphicx}
\usepackage{subcaption}
\usepackage{pgfplots}
\usepackage{xcolor}
\usepackage{cancel}
\usepackage{hyperref}
\usepackage[top=35mm, bottom=25mm, left=20mm, right=20mm]{geometry}
\usepackage{comment}

\providecommand{\U}[1]{\protect \rule{.1in}{.1in}}
\newtheorem{theorem}{Theorem}

\newtheorem{corollary}[theorem]{Corollary}
\newtheorem{fact}[theorem]{Fact}

\newtheorem{definition}[theorem]{Definition}
\newtheorem{example}[theorem]{Example}

\newtheorem{lemma}[theorem]{Lemma}

\newtheorem{proposition}[theorem]{Proposition}
\newtheorem{remark}[theorem]{Remark}

\newtheorem{assumption}[theorem]{Assumption}

\newcommand\R{\mathbb{R}}
\newcommand\Rinf{\overline{\mathbb{R}}}
\newcommand\dom[1]{ \bs{{\rm dom}}(#1)} 
\newcommand\dist{ \bs{{\rm dist}}} 

\newcommand\ov[1]{\overline{#1}}
\newcommand{\xb}{\overline x}
\newcommand{\ip}[2]{\left\langle #1,#2\right\rangle}
\newcommand{\norm}[1]{\left\lVert #1 \right\rVert}
\newcommand\mb{\mathbf{B}}
\newcommand\bs[1]{\boldsymbol{#1}}
\newcommand{\eps}{\varepsilon}
\newcommand{\Jp}{J_p}
\newcommand\argmin[1]{\bs{\arg\min}_{#1}}
\newcommand\argmint[1]{\mathop{\bs{\arg\min}}\limits_{#1}}
\newcommand\Nz{\mathbb{N}_0}

\newcommand\inft[1]{\mathop{\bs{\inf}}\limits_{#1}}

\newcommand\maxt[1]{\mathop{\bs{\max}}\limits_{#1}}
\let\oldliminf\liminf
\renewcommand{\liminf}{\bs{\oldliminf}}
\let\oldlimsup\limsup
\renewcommand{\limsup}{\bs{\oldlimsup}}
\let\oldmax\max
\renewcommand{\max}{\bs{\oldmax}}
\let\oldmin\min
\renewcommand{\min}{\bs{\oldmin}}
\let\oldsup\sup
\renewcommand{\sup}{\bs{\oldsup}}
\let\oldinf\inf
\renewcommand{\inf}{\bs{\oldinf}}
\let\oldlim\lim
\renewcommand{\lim}{\bs{\oldlim}}

\title{{\bf Robust Learning Meets Quasar-Convex Optimization: Inexact High-Order Proximal-Point Methods}}
\author{Alireza Kabgani\thanks{Department of Mathematics, University of Antwerp, Antwerp, Belgium. Email: alireza.kabgani@uantwerp.be} 
\and
Felipe Lara\thanks{Instituto de Alta Investigación (IAI), Universidad de Tarapacá, Arica, Chile. Email: flarao@academicos.uta.cl}
\and
Masoud Ahookhosh\thanks{Department of Mathematics, University of Antwerp, Antwerp, Belgium. Email: masoud.ahookhosh@uantwerp.be},
}

\begin{document}


\maketitle

\begin{abstract}
Robust learning aims to maintain model performance under noise, corruption, and distributional shifts, which are prevalent in modern machine learning applications. This work shows that examples of robust learning problems can be formulated as (strongly) quasar-convex optimization problems, which admit a benign landscape with no saddle points. We then propose HiPPA, an inexact high-order proximal-point method that employs a model-value gap to control the inexactness of subproblem solutions. Notably, we prove global convergence of HiPPA to global minima and establish that it attains a (local) linear or superlinear convergence rate, depending on the regularization order and inexactness control. Our numerical experiments on robust feature-alignment distillation indicate strong empirical performance of HiPPA and results consistent with our theoretical findings.
\noindent 

{\small}

\medskip


\medskip

\end{abstract}

\section{Introduction}\label{sec:intro}
Robust learning is central to modern machine learning, aiming to train models that remain reliable under label noise, feature corruption, outliers, distribution shifts, and adversarial perturbations \cite{song2022learning}. Such challenges are pervasive, for example, in representation learning, where supervision and intermediate features may be imperfect, and standard empirical risk minimization, particularly with squared losses, often overfits corrupted observations. Robust methods address this through resilient loss functions (e.g., Huber, capped, ramp, trimmed, redescending losses \cite{Barron2019AdaptiveRobustLoss,Brooks2011Ramp,Huber1964,Wang2019CappedL1}), regularization, and distributionally robust optimization \cite{duchi2019variance}. These approaches are critical in high-stakes domains such as healthcare and in areas like vision and natural language processing (NLP), where noise and adversarial effects can undermine reliability. 
However, these models often give rise to nonsmooth and nonconvex objectives, making optimization more challenging and weakening both theoretical guarantees and empirical quality.

Identifying nonconvex function classes with favorable optimization landscapes is a key challenge in modern machine learning and data science. Among such classes, the class of {\it (strongly) quasar-convex functions} arises in the modeling of a wide range of problems across several application domains, including learning linear dynamical systems \cite{Hardt2018Gradient}, phase retrieval \cite{wang2023continuized}, empirical risk minimization \cite{lee2016optimizing}, machine learning \cite{farzin2025minimization,wang2023continuized}, and beyond. Although several studies have been carried out in the smooth setting (e.g., \cite{guminov2023accelerated,Hardt2018Gradient,Hermant2024Study,Hinder2020Near,wang2023continuized}), relatively little attention has been paid to the nonsmooth setting (e.g., \cite{Ahookhosh2026Quasar,brito2025extending}). This motivates a more detailed study of this class of functions. Notably, such functions often exhibit a benign optimization landscape, characterized by the absence of spurious local minima and saddle points, which enables efficient global optimization despite nonconvexity, cf. \cite{Ahookhosh2026Quasar}, which is favorable in machine learning applications. Accordingly, a natural question arises in the optimization and machine learning communities:

``{\it How can we design optimization algorithms that are both theoretically strong and computationally efficient for (strongly) quasar-convex objectives arising in robust learning?}''

To address this question, we first identify instances of robust learning in which the objective functions are (strongly) quasar-convex, which is defined as follows:
\begin{definition}[\textbf{(Strong) quasar-convexity}]\label{def:qcx}
Let $\kappa\in(0,1]$, $\gamma\ge 0$, and let $h:\R^n\to \Rinf$ be proper with convex domain and $X^\star:=\argmin{x\in\R^n}h(x)\neq\emptyset$. We say that $h$ is $(\kappa,\gamma)$-\textit{strongly quasar-convex} with respect to $\xb\in X^\star$ if, for every $x\in\dom h$ and every $\lambda\in[0,1]$,
\begin{equation}\label{eq:qcx-def}
    h\left(\lambda \xb +(1-\lambda)x\right) \le \kappa\lambda h(\xb)+(1-\kappa\lambda)h(x)
 -\lambda\left(1-\tfrac{\lambda}{2-\kappa}\right)\tfrac{\kappa\gamma}{2}\norm{x-\xb}^2.
\end{equation}
When $\gamma=0$, we say that $h$ is $\kappa$-\textit{quasar-convex} with respect to $\xb$.
\end{definition}
We then develop an inexact proximal-point method with a more flexible regularization term, given by
\begin{equation}\label{eq:ppa}
    x^{k+1}\in\argmint{y\in\R^n}\left\{h(y)+\tfrac{1}{p\beta}\norm{y-x}^{p}\right\},
\end{equation}
with $\beta>0$ and $p>1$, which is referred to as HiPPA. The key components of these two steps are summarized below, and they constitute the main contributions of this paper:
\begin{itemize}[leftmargin=*]
    \item {\bf (Robust learning meets strong quasar-convexity)} We show that instances of robust learning objectives satisfy strong quasar-convexity, including robust feature-alignment distillation and anisotropic robust model stitching (see Examples~\ref{ex:capped-feature-distillation} and \ref{ex:anisotropic-model-stitching}). Consequently, these problems exhibit a benign landscape (all local minima are global) with no saddle points (i.e., the objective accepts only an isolated global solution), along with an error bound, making this class of large-scale problems well-suited for scalable first-order optimization (See Section~\ref{sec:Motivationa_examples}).
    \item {\bf (Inexact high-order proximal-point methods)} We design an inexact high-order proximal-point method (HiPPA, Algorithm~\ref{alg:ihippa}) for minimizing (strongly) quasar-convex functions. In particular, we tailor a model-value gap criterion that controls the level of inexactness, which is highly flexible and user-friendly for practical use, i.e., resulting in an inexact HiPPA with a simple structure and low-memory requirements.    
    We first show that the sequence generated by HiPPA converges to a global minimizer of the objective. Notably, if $\gamma>0$ (i.e., strongly quasar-convex case), then, for $p\in (1,2]$, HiPPA exhibits local linear convergence with complexity $\mathcal{O}(\log(\varepsilon^{-1}))$, and for $p>2$, it achieves superlinear convergence with complexity $\mathcal{O}(\log(\log(\varepsilon^{-1})))$. Finally, our numerical experiments demonstrate the strong empirical performance of HiPPA, supporting the theoretical results (see Section~\ref {sec:inexatHiPPA}). To our knowledge, it is the first inexact proximal-point method tailored to quasar-convex problems, combining strong convergence with efficient performance.
\end{itemize}

\subsection{Related work}
{\bf Quasar-convex optimization.} 
Quasar-convex functions arising as objectives in many practical applications are inherently nonsmooth, as nonsmoothness naturally emerges from regularization (e.g., sparsity-inducing penalties) and robust loss formulations designed to mitigate the effect of noise and outliers; see, e.g., \cite{Ahookhosh2026Quasar,brito2025extending}. 
They generalize convexity while remaining a broad subclass of nonconvex functions that includes star-convex functions \cite{Ahookhosh2026Quasar,guminov2023accelerated,Hermant2024Study,Hinder2020Near,nesterov2006cubic}. They enjoy useful properties, such as first-order characterizations, calculus rules, error bounds, and the Polyak-{\L}ojasiewicz inequality, in both smooth \cite{guminov2023accelerated,Hardt2018Gradient,Hermant2024Study,Hinder2020Near,Pun2024,wang2023continuized} and nonsmooth settings \cite{Ahookhosh2026Quasar,brito2025extending}. Notably, they have a benign landscape, every local minimum is global, and no saddle points \cite[Section~3.3]{Ahookhosh2026Quasar}. Moreover, $(\kappa,\gamma)$-strong quasar-convexity with $\gamma>0$ ensures isolated global minimizers, making this class well-suited for first-order methods.
Subgradient methods \cite{davis2018subgradient,rahimi2024projected} and inexact proximal-point schemes \cite{brito2025extending,sra2012scalable} are standard optimization methods for such nonsmooth and nonconvex problems. Although subgradient methods are simple and general, they often converge slowly and are sensitive to conditioning; proximal-point methods are typically more stable and faster via implicit regularization.

{\bf Proximal-point methods and their inexact and high-order variants.}
The classical proximal-point method (PPA, \eqref{eq:ppa} with $p=2$) forms a powerful framework for nonsmooth optimization, based on iteratively solving regularized subproblems that stabilize the original objective. Originating in monotone operator theory, they have strong convergence guarantees even in nonsmooth and nonconvex settings, and can be interpreted as implicit gradient steps that improve robustness to conditioning; see, e.g., \cite{martinet1970regularisation,martinet1972determination,rockafellar1976monotone}. 
In practice, however, the resulting proximal subproblems are often too expensive to solve exactly. Inexact proximal-point methods address this issue by allowing approximate solutions under controlled accuracy requirements, thereby reducing the per-iteration computational burden while maintaining convergence guarantees; cf. \cite{barre2023principled,Dvurechenskii2022,liu2012implementable,rockafellar1976monotone,Salzo12,Solodov01}. 
This flexibility makes them particularly suitable for large-scale settings, where exact proximal evaluations are impractical, and forms the basis for several scalable first-order and operator-splitting schemes. 
High-order proximal-point methods and their inexact variants have recently attracted significant attention for their ability to capture detailed local geometry of the objective, cf. \cite{Ahookhosh2026Quasar,ahookhosh2025asymptotic,ahookhosh2024high,ahookhosh2025high,kabgani2025moreau,kabgani2024itsopt,kabgani2025itsdeal,kabgani2025fundamental,kabgani2025first,nesterov2023inexact}. This added flexibility allows for more effective steps, improved complexity, and a better balance between global regularization and local model fidelity, making high-order methods a natural and powerful proximal framework for modern large-scale optimization.


\section{Preliminaries and notation}\label{sec:preliminaries}
Throughout this paper, $\R^n$ denotes the $n$-dimensional \textit{Euclidean space} equipped with the \textit{inner product} $\langle \cdot,\cdot \rangle$ and the associated norm $\|\cdot\| := \sqrt{\langle \cdot,\cdot \rangle}$. 
The \textit{open ball} centered at $\ov{x}\in \R^n$ with radius $r>0$ 
is denoted by $\mb(\ov{x}, r)$.
The set of \textit{nonnegative integers} is denoted by $\Nz := \mathbb{N}\cup\{0\}$. 

Let $h: \mathbb{R}^{n} \rightarrow \overline{\mathbb{R}} := \mathbb{R}\cup\{+\infty\}$ be an extended-valued function. 
The \textit{effective domain} of $h$ is defined by
$\dom{h} := \{x \in \mathbb{R}^{n} : h(x) < +\infty\}$.
The function $h$ is said to be \textit{proper} if $\dom{h} \neq \emptyset$. 
We denote by $X^\star:=\argmin{\mathbb{R}^{n}} h$ the set of minimizers of $h$ and the minimal value by $h^\star$.
A function $h$ is \textit{lower semicontinuous} (lsc henceforth) at $\ov{x} \in \mathbb{R}^{n}$ if for every sequence $\{x^k\}_{k\in\Nz} \subseteq \mathbb{R}^{n}$ with $x^k \to \ov{x}$, one has
$h(\ov{x}) \leq \liminf_{k \to +\infty} h(x^k)$. 
A point $\widehat{x}\in\R^n$ is called a \textit{cluster point} of a sequence $\{x^k\}_{k\in\Nz}$ if there exists an infinite subset $J\subseteq\Nz$ such that $x^j\to \widehat{x}$ as $j\to \infty$ with $j\in J$.

Let $h:\R^n \to \overline{\R}$ be proper and locally Lipschitz around $x\in\dom h$.
The Clarke directional derivative of $h$ at $x$ in direction $d \in \mathbb{R}^n$ is defined by
\[
h^\circ (x; d) :=\mathop{\limsup}\limits_{\substack{y \to x \\ t\downarrow 0}}\frac{h(y+td) - h(y)}{t},
\]
 and the Clarke subdifferential of $h$ at $x$ is given by
 \begin{equation*}
  \partial^{C} h(x) := \{v \in \mathbb R^n: \, \langle v, d \rangle \le h^\circ(x; d), ~ \forall ~ d \in \mathbb R^n\},
 \end{equation*} 
By convention, $\partial^C h(x):=\emptyset$ for $x\notin\dom h$. Moreover,
\begin{equation}\label{prop:direc}
 h^{\circ} (x; d) = \maxt{\xi \in \partial^{C} h(x)} \langle \xi, d \rangle, \qquad \forall ~ d \in \mathbb R^{n}.
\end{equation} 
Moreover, we have
\[
\partial^C h(x)=\operatorname{cl}\operatorname{conv}\left\{v\in\mathbb R^n:\exists\, x^j\to x,\ h 
\text{ is differentiable at } x^j,\ \nabla h(x^j)\to v\right\},
\]
where $\operatorname{cl}\operatorname{conv}$ denotes the \textit{closed convex hull}. If, in addition, $h$ is convex, then the Clarke subdifferential coincides with the standard convex subdifferential.
 A point $x \in \dom h$ is called a Clarke critical point of $h$ if $0 \in \partial^{C} h(x)$.
We use the standard fact that, for locally Lipschitz functions, the Clarke subdifferential is locally bounded and outer semicontinuous; see \cite{Clarke1990,Rockafellar09}.

A function $h: \R^{n} \rightarrow \Rinf$ with convex domain is said to be:
\begin{enumerate}[label=(\textbf{\alph*}), font=\normalfont\bfseries, leftmargin=0.7cm]
\item\label{def:st:conv} \textit{strongly convex} on $\dom h$ if there exists $\gamma > 0$ such that, for all $x,y \in \dom h$ and all $\lambda \in [0,1]$,
 \begin{equation}\label{strong:convex}
  h(\lambda y + (1-\lambda)x) \leq \lambda h(y) + (1-\lambda) h(x) - \lambda (1 - \lambda) \frac{\gamma}{2} \lVert x - y \rVert^{2},
 \end{equation}

\item\label{def:st:star} \textit{strongly star-convex} with respect to $\ov{x}\in\argmin{\mathbb{R}^n} h$ if there exists $\gamma>0$ such that, for all $x\in\dom h$ and all $\lambda\in[0,1]$,
 \begin{equation}\label{strong:star-convex}
  h(\lambda \ov{x}+(1-\lambda)x) \leq \lambda h(\ov{x})+(1-\lambda)h(x)-\lambda(1-\lambda)\frac{\gamma}{2}\lVert x-\ov{x}\rVert^2.
 \end{equation}
\end{enumerate}
Setting $\gamma=0$ in \eqref{strong:convex} (resp. \eqref{strong:star-convex}) yields the notions of \textit{convexity} (resp. \textit{star-convexity}).
In particular, every (strongly) convex function is (strongly) star-convex with respect to any point $\ov{x}\in\dom h$, whereas the converse implication does not hold in general (see \cite{hadjisavvas2006handbook,Lara2022strongly,nesterov2006cubic}).

\subsection{Further properties of quasar-convexity}

We first recall a first-order characterization of strong quasar-convexity in terms of the Clarke subdifferential. This form is useful because it converts the geometric definition into inequalities involving generalized gradients.

\begin{fact}[{\bf Clarke first-order characterization}]\label{fact:char:clarke}\cite[Theorem~31]{Ahookhosh2026Quasar}
 Let $h: \mathbb{R}^{n} \to \overline{\mathbb{R}}$ be a proper function with convex domain, and let $\overline{x} \in \argmin{\mathbb R^n} h$. Assume that $h$ is locally Lipschitz on $\dom{h}$. Then $h$ is $(\kappa, \gamma)$-strongly quasar-convex with respect to $\overline{x}$, where $\gamma \ge 0$, if and only if
\begin{equation}\label{eq:foC}
 h(\overline x) \ge h(x) + \frac{1}{\kappa} \langle v, \overline x-x\rangle + \frac{\gamma}{2} \|x-\overline x\|^2, \qquad \forall ~ x \in \dom{h}, ~  \forall ~ v \in \partial^C h(x).
\end{equation}
\end{fact}

The next consequence explains why the quasar-convexity guarantees a benign nonsmooth landscape: Clarke criticality cannot occur away from the set of global minimizers.
\begin{fact}[\textbf{Clarke critical points are global minimizers}]\label{crit:are:glob}\cite[Proposition~33]{Ahookhosh2026Quasar}
 Let $h:\mathbb{R}^n \to \overline{\mathbb{R}}$ be a proper function with convex domain, and let $\overline{x} \in \argmin{\mathbb{R}^n} h$. Assume that $h$ is $(\kappa,\gamma)$-strongly quasar-convex with respect to $\overline{x}$, where $\gamma \ge 0$, and that $h$ is locally Lipschitz on $\dom h$. If $0 \in \partial^C h(x)$, then $x \in \argmin{\mathbb{R}^{n}} h$.
\end{fact}

When $\gamma>0$, strong quasar-convexity also provides quantitative control near the minimizer. The subsequent growth condition and error-bound estimate will be used later to convert residual information into distance and function-value bounds.
\begin{fact}[{\bf Quadratic growth and error bound}]\label{fact:error-bounds}\cite[Corollary~32]{Ahookhosh2026Quasar}
Let $h:\mathbb{R}^n \to \overline{\mathbb{R}}$ be a proper function and let $\overline x \in \argmin{\mathbb{R}^n} h$. Assume that $h$ is $(\kappa, \gamma)$-strongly quasar-convex  with respect to $\overline{x}$, where $\kappa \in (0,1]$ and $\gamma>0$, and that $h$ is locally Lipschitz on $\dom h$. Then, for every $x \in \dom{h}$, it holds that
\begin{equation}\label{eq:quad-growth}
 h(x)-h(\xb)\ge \frac{\kappa\gamma}{2(2-\kappa)}\norm{x-\xb}^2,
\end{equation}
and
\begin{equation}\label{eq:error_bound}
 \norm{x-\xb}\le \frac{2}{\kappa\gamma}\dist(0,\partial^C h(x)).
\end{equation}
\end{fact}

The next estimate is the main tool for converting an approximate stationarity condition into quantitative control of both the distance to the minimizer and the objective gap. In particular, it will be used with $w$ equal to the proximal regularization term and $r$ equal to the residual error of an inexact proximal step.

\begin{proposition}[\textbf{Residual-to-gap inequality}]\label{prop:residual-gap}
Let $h:\mathbb R^n\to\overline{\mathbb R}$ be proper, locally Lipschitz on $\dom h$, and $(\kappa,\gamma)$-strongly quasar-convex with respect to $\xb\in\argmin{\mathbb R^n} h$. 
Let $x\in\dom h$ and suppose that
\begin{equation}\label{eq:residual-inclusion-static}
 r\in \partial^C h(x)+w,
\end{equation}
for some $r,w\in\R^n$. Then,
\begin{equation}\label{eq:residual-gap}
 \kappa\left(h(x)-h(\xb)\right)+\frac{\kappa\gamma}{2}\norm{x-\xb}^2 \le \ip{w}{\xb-x}+\norm{r}\,\norm{x-\xb}.
\end{equation}
If moreover $\gamma>0$, then
\begin{equation}\label{eq:residual-distbound}
 \norm{x-\xb}\le \frac{2}{\kappa\gamma}\left(\norm{w}+\norm{r}\right).
\end{equation}
Furthermore,
\begin{equation}\label{eq:residual-gap-loose}
 h(x)-h(\xb) \le \frac{\norm{w}+\norm{r}}{\kappa}\,\norm{x-\xb} \le \frac{2}{\kappa^2\gamma}\left(\norm{w}+\norm{r}\right)^2.
\end{equation}
Moreover, the sharper bound
\begin{equation}\label{eq:residual-gap-sharp}
 h(x)-h(\xb) \le \frac{1}{2\kappa^2\gamma}\left(\norm{w}+\norm{r}\right)^2
\end{equation}
holds.
\end{proposition}
\begin{proof}
Let us choose $v\in\partial^C h(x)$ such that $r=v+w$. It follows from \eqref{eq:foC} that
\[
 h(\xb)\ge h(x)+\frac{1}{\kappa}\ip{v}{\xb-x}+\frac{\gamma}{2}\norm{x-\xb}^2.
\]
Multiplying by $\kappa$ and rearranging, it can be concluded that
\[
 \kappa\left(h(x)-h(\xb)\right)+\frac{\kappa\gamma}{2}\norm{x-\xb}^2 \le \ip{v}{x-\xb}.
\]
Substituting $v=r-w$ in this inequality results in
\[
 \kappa\left(h(x)-h(\xb)\right)+\frac{\kappa\gamma}{2}\norm{x-\xb}^2 \le \ip{r-w}{x-\xb} = \ip{r}{x-\xb}+\ip{w}{\xb-x}.
\]
Using the Cauchy--Schwarz inequality on the first term yields \eqref{eq:residual-gap}.
Now, let us assume $\gamma>0$. Since $h(x)\ge h(\xb)$, the first term on the left-hand side of
\eqref{eq:residual-gap} is nonnegative. Hence
\[
 \frac{\kappa\gamma}{2}\norm{x-\xb}^2 \le \ip{w}{\xb-x}+\norm{r}\,\norm{x-\xb} \le \left(\norm{w}+\norm{r}\right)\norm{x-\xb},
\]
If $x=\xb$, the claim is trivial. Otherwise, dividing by $\norm{x-\xb}$ gives \eqref{eq:residual-distbound}.
Dropping the quadratic term in \eqref{eq:residual-gap} and applying the Cauchy--Schwarz inequality gives
\[
 \kappa\left(h(x)-h(\xb)\right) \le \left(\norm{w}+\norm{r}\right)\norm{x-\xb},
\]
which implies the first inequality in \eqref{eq:residual-gap-loose}. Combining this estimate with
\eqref{eq:residual-distbound} yields the second inequality in \eqref{eq:residual-gap-loose}.
For the sharper quadratic bound, set $d:=\norm{x-\xb}$ and $M:=\norm{w}+\norm{r}$. From \eqref{eq:residual-gap} and the Cauchy--Schwarz inequality, we obtain
\[
 \kappa\left(h(x)-h(\xb)\right)+\frac{\kappa\gamma}{2}d^2 \le Md,
\]
i.e.,
\[
 \kappa\left(h(x)-h(\xb)\right)\le Md-\frac{\kappa\gamma}{2}d^2.
\]
The right-hand side is maximized over $d\ge 0$ at $d=M/(\kappa\gamma)$, and the maximal value is $M^2/(2\kappa\gamma)$, i.e.,
\[
 h(x)-h(\xb) \le \frac{1}{2\kappa^2\gamma}\left(\norm{w}+\norm{r}\right)^2,
\]
which consequently adjusts the inequality \eqref{eq:residual-gap-sharp}.
\end{proof}

The previous estimate applies to any approximate stationarity condition. We now specialize it to the residual generated by a high-order proximal subproblem. This
gives a one-step stability inequality that links the proximal residual, the regularization displacement $y-x$, and the distance from the new point $y$ to the reference minimizer.
\begin{proposition}[\textbf{Stability of approximate proximal points}]\label{prop:stability-approx-prox}
Let $h:\mathbb R^n\to\overline{\mathbb R}$ be proper, locally Lipschitz on $\dom h$, and $(\kappa,\gamma)$-strongly quasar-convex with respect to $\xb\in\argmin{\mathbb R^n}h$. Fix $x\in\mathbb R^n$, $\beta>0$, and $p>1$. Suppose that $y\in\dom h$ satisfies
\begin{equation}\label{eq:prox-residual-static}
 e\in \partial^C h(y)+\frac{1}{\beta}\norm{y-x}^{p-2}(y-x),
\end{equation}
for some $e\in\R^n$. Then,
\begin{equation}\label{eq:stability-prox-signed}
 \kappa\left(h(y)-h(\xb)\right)+\frac{\kappa\gamma}{2}\norm{y-\xb}^2 \le \frac{1}{\beta}\norm{y-x}^{p-2}\ip{y-x}{\xb-y}+\norm{e}\,\norm{y-\xb}.
\end{equation}
Moreover,
\begin{equation}\label{eq:stability-prox}
 \kappa\left(h(y)-h(\xb)\right)+\frac{\kappa\gamma}{2}\norm{y-\xb}^2 \le \frac{1}{\beta}\norm{x-y}^{p-1}\norm{y-\xb}+\norm{e}\,\norm{y-\xb}.
\end{equation}
\end{proposition}
\begin{proof}
Let us choose $v\in\partial^C h(y)$ such that
\[
 e=v+\frac{1}{\beta}\norm{y-x}^{p-2}(y-x).
\]
Applying the Clarke characterization \eqref{eq:foC} at $y$ and rearranging ensures
\[
 \kappa\left(h(y)-h(\xb)\right)+\frac{\kappa\gamma}{2}\norm{y-\xb}^2 \le \ip{v}{y-\xb}.
\]
Substituting $ v=e-\tfrac{1}{\beta}\norm{y-x}^{p-2}(y-x)$ into the last inequality, we come to
\[
 \kappa\left(h(y)-h(\xb)\right)+\frac{\kappa\gamma}{2}\norm{y-\xb}^2 \le \ip{e}{y-\xb}+\frac{1}{\beta}\norm{y-x}^{p-2}\ip{y-x}{\xb-y},
\]
Applying the Cauchy-Schwarz inequality to the first term leads to \eqref{eq:stability-prox-signed}.
Employing the Cauchy-Schwarz inequality again on the proximal term results in \eqref{eq:stability-prox}.
\end{proof}

\section{Motivational examples from robust learning}
\label{sec:Motivationa_examples}
In this section, we study two nonsmooth and nonconvex applications arising in robust learning.
We show that they belong to the class of strongly quasar-convex functions while failing to be star-convex. These two problems are: (i) {\it robust feature-alignment distillation} and (ii) {\it anisotropic robust model stitching}. They are built from two robust alignment templates: a stabilized capped loss, inspired by capped and truncated losses used to reduce the influence of outliers, and a weighted kinked radial loss, which permits direction-dependent penalties in representation space. After composition with linear or affine alignment maps, these templates produce robust alignment objectives motivated by neural representation learning. The resulting functions are locally Lipschitz, nonsmooth, nonconvex, and non-star-convex, yet retain a one-point geometry toward a reference minimizer.

\subsection{Stabilized capped residual loss}
The capped and truncated losses are widely used to reduce the influence of large outliers. For example, the capped $L_1$-type losses have been used in robust SVM-type formulations, where they lead to nonsmooth and nonconvex optimization problems. Similarly, the ramp loss is a standard robust nonconvex classification loss; see \cite{Huang2014RampLossSVM,Yang2023CappedL1PTSVM}. More broadly, adaptive robust losses in vision interpolate among classical losses such as Charbonnier, pseudo-Huber, Cauchy, Geman--McClure, Welsch, and squared losses; see \cite{Barron2019AdaptiveRobustLoss}. Motivated by these robust-loss mechanisms, we introduce the following stabilized master families.

Let us fix $\tau,\mu>0$ and define $\Phi_{\tau,\mu}:\R^d\to\R$ as
\begin{equation}\label{eq:capped_residual}
\Phi_{\tau,\mu}(z):=\min\{\|z\|,\tau\}+\mu\|z\|^2.
\end{equation}
The capped term limits the effect of large residuals, while the quadratic term provides a global stabilizing pull toward the origin. The following result shows that this standard robust-loss mechanism, after stabilization, yields a locally Lipschitz, nonsmooth, nonconvex, and non-star-convex objective that is still strongly quasar-convex.
\begin{proposition}[\textbf{Stabilized capped residual loss}]\label{prop:stabilized-capped-loss}
Let $\tau,\mu>0$ satisfy $0<\mu\tau<1/2$. Then $\Phi_{\tau,\mu}$ defined in \eqref{eq:capped_residual} is locally Lipschitz,  nonsmooth, and non-star-convex with respect to the origin. Moreover, it is $(\kappa,\gamma)$-strongly quasar-convex with respect to $\ov z=0$, where one may take $\gamma=\mu$ and $\kappa=\frac{\mu\tau}{1+2\mu\tau}$.
\end{proposition}
\begin{proof}
Local Lipschitz continuity follows from the Lipschitz continuity of $z\mapsto \|z\|$, the stability of locally Lipschitz functions under a pointwise minimum, and the smoothness of $z\mapsto \mu\|z\|^2$. The function is nonsmooth at $\ov z=0$ and along the sphere $\{z:\|z\|=\tau\}$. 
We next show that $\Phi_{\tau,\mu}$ is not star-convex with respect to $\ov z=0$. Fix a unit vector $u$ and set $z=2\tau u$. Then
\[
\Phi_{\tau,\mu}(\tau u)=\tau+\mu\tau^2,
\]
whereas
\[
\frac{1}{2}\Phi_{\tau,\mu}(2\tau u)=\frac{1}{2}(\tau+4\mu\tau^2)=\frac{\tau}{2}+2\mu\tau^2.
\]
Since $\mu\tau<1/2$, we have $\tau+\mu\tau^2>\frac{\tau}{2}+2\mu\tau^2$.
Thus
\[
\Phi_{\tau,\mu}\left(\frac{1}{2} z\right)>\frac{1}{2}\Phi_{\tau,\mu}(z),
\]
which violates the star-convexity with respect to $\ov z = 0$. 
It remains to prove strong quasar-convexity. We use Clarke's first-order characterization \eqref{eq:foC}. 
We claim that, with $\gamma=\mu$ and $\kappa=\frac{\mu\tau}{1+2\mu\tau}$, for every $z\in\R^d$ and every $v\in\partial^C\Phi_{\tau,\mu}(z)$,
\begin{equation}\label{eq1:prop:stabilized-capped-loss}
0\ge\Phi_{\tau,\mu}(z)+\frac1\kappa\langle v,-z\rangle+\frac{\gamma}{2}\|z\|^2.
\end{equation}
For $z=0$, this inequality is immediate. 
Let $z\neq0$ be arbitrary. Set $r:=\|z\|$, and write $u=z/r$.
If $0<r<\tau$, then $\Phi_{\tau,\mu}(z)=\|z\|+\mu\|z\|^2$ is continuously differentiable at $z$, and $\partial^C\Phi_{\tau,\mu}(z)=\{v_1:=(1+2\mu r)u\}$.
Since $0<\kappa\le1$,
\[
\langle v_1,z\rangle = r+2\mu r^2  \ge \kappa(r+\mu r^2)+\frac{\kappa\mu}{2}r^2=\kappa \left(\Phi_{\tau,\mu}(z)+\frac{\gamma}{2}\|z\|^2\right).
\]
Hence, the inequality \eqref{eq1:prop:stabilized-capped-loss} holds in this case.
If $r>\tau$, then $\Phi_{\tau,\mu}(z)=\tau+\mu\|z\|^2$ is continuously differentiable at $z$, and
$\partial^C\Phi_{\tau,\mu}(z)=\{v_2:=2\mu r u\}$, and consequently we come to
\[
\langle v_2,z\rangle=2\mu r^2.
\]
It remains to consider the kink $r=\tau$. In this case,
the Clarke subdifferential is the convex hull of the two radial limits, i.e., $\partial^C\Phi_{\tau,\mu}(z)=[2\mu r^2, r+2\mu r^2]$, i.e., the smallest possible value of $\langle v,z\rangle$ is
$2\mu\tau^2$. Consequently, it is sufficient to verify the estimate for $r\ge\tau$ using
the lower radial slope. 
We have
\[
\mu\tau\left(2-\frac{3\kappa}{2}\right)\ge \frac{\mu\tau}{1+2\mu\tau}=\kappa.
\]
Since $r\ge\tau$, this implies
\[
\mu\left(2-\frac{3\kappa}{2}\right)r^2\ge\kappa\tau.
\]
Consequently,  for every $v\in\partial^C\Phi_{\tau,\mu}(z)$ with $r\ge\tau$,
\[
\langle v,z\rangle \ge 2\mu r^2\ge \kappa(\tau+\mu r^2)+\frac{\kappa\mu}{2}r^2=\kappa \left(\Phi_{\tau,\mu}(z)+\frac{\gamma}{2}\|z\|^2\right).
\]
Equivalently, \eqref{eq1:prop:stabilized-capped-loss} hold. By Fact~\ref{fact:char:clarke}, $\Phi_{\tau,\mu}$ is $(\kappa,\gamma)$-strongly quasar-convex with respect to $\ov z =0$.
\end{proof}

\begin{example}[\textbf{Robust feature-alignment distillation}]
\label{ex:capped-feature-distillation}
Feature-based knowledge distillation trains a student model using intermediate
teacher representations as supervision; FitNets is a classical example in which
teacher hidden representations are used as hints for a thinner student
network~\cite{Romero2015FitNets}. Motivated by this setting, let
$Z_s=[z_{s,1},\ldots,z_{s,N}]\in\mathbb R^{d_s\times N}$ and
$Z_t=[z_{t,1},\ldots,z_{t,N}]\in\mathbb R^{d_t\times N}$ denote student and
teacher features collected from the same batch, respectively. For a trainable alignment map
$W\in\mathbb R^{d_t\times d_s}$, we define
\begin{equation}\label{eq:kd-loss}
h_{\rm KD}(W):=\frac1N\sum_{i=1}^N\Phi_{\tau,\mu}(Wz_{s,i}-z_{t,i}).
\end{equation}
The capped term reduces the influence of corrupted or mismatched teacher
features, while the quadratic term preserves the strong quasar-convex geometry. 
Assume that $Z_s$ has full row rank, that there exists $W^\star\in\mathbb R^{d_t\times d_s}$ such that $W^\star Z_s=Z_t$, and that $0<\mu\tau<1/2$. 
Let $\sigma_{\min}(Z_s)$ denote the smallest singular value of $Z_s$.
Then $h_{\rm KD}$ is locally Lipschitz, nonsmooth, and $(\kappa, \gamma)$-strongly quasar-convex with respect to $W^\star$, where
$\kappa=\frac{\mu\tau}{1+2\mu\tau}$ and $\gamma=\frac{\mu}{N}\sigma_{\min}(Z_s)^2$. 
\end{example}
\begin{proof}
Let us define
\[
R_i(W):=Wz_{s,i}-z_{t,i}, \qquad i=1,\ldots,N.
\]
By the realizability assumption $W^\star Z_s=Z_t$, we have $R_i(W^\star)=0$.
Moreover, 
\[
R_i(W)= (W-W^\star)z_{s,i}.
\]
The local Lipschitz continuity of $h_{\rm KD}$ follows from Proposition~\ref{prop:stabilized-capped-loss}. The nonsmoothness comes from the fact that $R_i(W^\star)=0$ for all $i$ and the residual loss $\Phi_{\tau,\mu}$ is nonsmooth at the origin. In particular, along any direction $\Delta W$ with $\Delta WZ_s\neq 0$, the restriction $t\mapsto h_{\rm KD}(W^\star+t\Delta W)$ has a kink at $t=0$.

It remains to verify the strong quasar-convexity. 
Let $W\in\mathbb R^{d_t\times d_s}$ and $\lambda\in[0,1]$, and set 
\[
W_\lambda:=\lambda W^\star+(1-\lambda)W.
\]
For every $i$, using $W^\star z_{s,i}=z_{t,i}$, it holds that
\[
R_i(W_\lambda)=W_\lambda z_{s,i}-z_{t,i}=(1-\lambda)(Wz_{s,i}-z_{t,i})=(1-\lambda)R_i(W).
\]
By Proposition~\ref{prop:stabilized-capped-loss}, $\Phi_{\tau,\mu}$ is $(\kappa,\mu)$-strongly quasar-convex with respect to
$\ov z=0$, where $\kappa=\frac{\mu\tau}{1+2\mu\tau}$. Therefore, for each $i$,
\[
\Phi_{\tau,\mu}(R_i(W_\lambda))\le(1-\kappa\lambda)\Phi_{\tau,\mu}(R_i(W))-\lambda\left(1-\frac{\lambda}{2-\kappa}\right)\frac{\kappa\mu}{2}\|R_i(W)\|^2.
\]
Averaging over $i=1,\ldots,N$ ensures
\[
h_{\rm KD}(W_\lambda)\le(1-\kappa\lambda)h_{\rm KD}(W)-\lambda\left(1-\frac{\lambda}{2-\kappa}\right)\frac{\kappa\mu}{2N}\sum_{i=1}^N\|R_i(W)\|^2.
\]
Since $R_i(W)=(W-W^\star)z_{s,i}$, it can be concluded that
\[
\sum_{i=1}^N\|R_i(W)\|^2=\|(W-W^\star)Z_s\|_F^2.
\]
In light of the fact that $Z_s$ has full row rank, we get
\[
\|(W-W^\star)Z_s\|_F^2\ge \sigma_{\min}(Z_s)^2\|W-W^\star\|_F^2,
\]
which consequently implies
\[
h_{\rm KD}(W_\lambda)\le(1-\kappa\lambda)h_{\rm KD}(W)-\lambda\left(1-\frac{\lambda}{2-\kappa}\right)\frac{\kappa}{2}\left(\frac{\mu}{N}\sigma_{\min}(Z_s)^2
\right)\|W-W^\star\|_F^2.
\]
Since $h_{\rm KD}(W^\star)=0$, this is precisely $\left(\kappa,\frac{\mu}{N}\sigma_{\min}(Z_s)^2\right)$-strong quasar-convexity with respect to $W^\star$.
\end{proof}

\begin{remark}[\textbf{Robust feature-alignment distillation is not star-convex}]
\label{rem:kd-nonstar}
Proposition~\ref{prop:stabilized-capped-loss} shows that the residual loss
$\Phi_{\tau,\mu}$ is not star-convex with respect to the origin. This property
does not automatically transfer to the averaged composed objective
$h_{\rm KD}$, because different residual terms may lie in different regions of the capped loss and the one-dimensional violation can be averaged out.

A simple sufficient condition is the existence of a direction
$\Delta W\in\mathbb R^{d_t\times d_s}$ such that
\begin{equation}\label{eq:kd-nonstar-condition}
\sum_{i=1}^N\left[\Phi_{\tau,\mu}\!\left(\frac12\Delta W z_{s,i}\right)-\frac12\Phi_{\tau,\mu}(\Delta W z_{s,i})\right]>0.
\end{equation}
Under this condition, $h_{\rm KD}$ is not star-convex with respect to
$W^\star$. Indeed, assume that there exists $\Delta W\in\mathbb R^{d_t\times d_s}$ satisfying \eqref{eq:kd-nonstar-condition}. Set
\[
W:=W^\star+\Delta W.
\]
Since $W^\star Z_s=Z_t$, we have
\[
Wz_{s,i}-z_{t,i} = \Delta W z_{s,i}, \qquad i=1,\ldots,N.
\]
Moreover,
\[
 \left(\frac12 W^\star+\frac12 W\right)z_{s,i}-z_{t,i}=\frac12\Delta W z_{s,i}.
\]
Therefore, from \eqref{eq:kd-nonstar-condition}
\[
h_{\rm KD}\left(\frac12 W^\star+\frac12 W\right)=\frac1N\sum_{i=1}^N\Phi_{\tau,\mu}\left(\frac12\Delta W z_{s,i}\right)>\frac1{2N}\sum_{i=1}^N\Phi_{\tau,\mu}(\Delta W z_{s,i})=\frac12 h_{\rm KD}(W).
\]
Since $h_{\rm KD}(W^\star)=0$, the star-convexity inequality with center
$W^\star$ would require
\[
h_{\rm KD}\!\left(\frac12 W^\star+\frac12 W\right)\le\frac12 h_{\rm KD}(W^\star)+\frac12 h_{\rm KD}(W)=\frac12 h_{\rm KD}(W),
\]
which contradicts the strict inequality above. Hence, $h_{\rm KD}$ is not star-convex with respect to $W^\star$.

\end{remark}

\subsection{Anisotropic robust residuals}
The capped loss \eqref{eq:capped_residual} is isotropic: it depends only on the residual norm and therefore treats all residual directions equally. However, in neural representation learning, residual directions may have different semantic roles, reliability levels, or alignment importance. This is especially natural in modern adaptation pipelines, where one often keeps large backbones or embedding models fixed and trains only lightweight task-specific components, such as low-rank adaptation matrices for frozen Transformer weights \cite{Hu2021LoRA}, feature-alignment maps between teacher and student representations \cite{Romero2015FitNets}, retrieval adapters that align query and document embedding spaces \cite{Maekawa2026Align}, or affine connectors used for model stitching across neural representations
\cite{Csiszarik2021RepresentationMatching}. Motivated by these settings, we next introduce an angularly weighted robust radial loss. Its radial profile controls nonsmoothness and nonconvex robustness, while the angular factor
$q(z/\|z\|)$ allows direction-dependent penalties in the representation residual.

\begin{proposition}[\textbf{A locally Lipschitz nonsmooth master family}]\label{prop:master-family}
Fix parameters $\tau,\mu>0$ such that $\mu\tau<\frac{1}{8}$ and define
\begin{equation}\label{eq:app-psi-tau}
 \psi_{\tau}(r):= \begin{cases}
 r, & 0\le r\le \tau,\\[0.5ex]
 \sqrt{\tau r}, & r>\tau.
 \end{cases}
\end{equation}
Let $f_{\tau,\mu}(r):=\psi_{\tau}(r)+\mu r^2$ for $r\ge 0$.
Let $d\ge 2$ and let $q:\mathbb S^{d-1}\to[1,\infty)$ be a nonconstant $C^1$ function, where 
$\mathbb S^{d-1}:=\{u\in\mathbb R^d:\|u\|=1\}$. 
Define
\[
H_{\tau,\mu,q}(z):=\begin{cases}
f_{\tau,\mu}(\norm z)\,q\left(\dfrac{z}{\norm z}\right), & z\neq 0,\\[1mm]
0,& z=0,
\end{cases}
\]
Then, $H_{\tau,\mu,q}$ is locally Lipschitz,  nonsmooth, and non-star-convex with respect to the origin. Moreover, it is $(\frac{1}{2},6\mu)$-strongly quasar-convex with respect to $\ov z=0$.
\end{proposition}
\begin{proof}
We first analyze the scalar radial core $f_{\tau,\mu}$. Since $\psi_{\tau}$ is continuous and piecewise $C^1$ with slope $1$ on $[0,\tau)$ and slope $\frac12\sqrt{\tau/r}\le \frac12$ on $(\tau,+\infty)$, it is globally Lipschitz on $[0,+\infty)$. Hence, $f_{\tau,\mu}$ is locally Lipschitz and nonsmooth at $r=\tau$. We claim that
\begin{equation}\label{eq:psi-shrink}
 \psi_{\tau}((1-\lambda)r)\le \left(1-\frac{\lambda}{2}\right)\psi_{\tau}(r),
 \qquad \forall r\ge 0,\ \forall \lambda\in[0,1].
\end{equation}
If $r\le \tau$, then $(1-\lambda)r\le \tau$ and \eqref{eq:psi-shrink} become $(1-\lambda)r\le (1-\lambda/2)r$, which is obvious. If $r>\tau$ and $(1-\lambda)r>\tau$, then
\[
 \psi_{\tau}((1-\lambda)r)=\sqrt{1-\lambda}\,\psi_{\tau}(r) \le \left(1-\frac{\lambda}{2}\right)\psi_{\tau}(r).
\]
It remains to consider the mixed case $r>\tau$ and $(1-\lambda)r\le \tau$.
Set $s:=1-\lambda\in[0,1]$ and $a:=\sqrt{r/\tau}>1$. The condition
$s r\le \tau$ is equivalent to $s\le a^{-2}$. We need to prove
\[
 sr\le \frac{1+s}{2}\sqrt{\tau r},
\]
or, equivalently,
\[
 sa\le \frac{1+s}{2}.
\]
The function $\varphi(s):=(1+s)/2-sa$ decreases on $[0,a^{-2}]$, because
$a>1$. Hence, its minimum on this interval is attained at $s=a^{-2}$ and for each $s\in [0, a^{-2}]$,
\[
\frac{1+s}{2}-sa=\varphi(s)\ge \varphi(a^{-2})=\frac{1+a^{-2}}{2}-\frac1a =\frac{(a-1)^2}{2a^2}\ge 0.
\]
Thus, we have $sa\le (1+s)/2$, which proves \eqref{eq:psi-shrink} in the mixed case and
therefore completes the proof of \eqref{eq:psi-shrink}.
Using \eqref{eq:psi-shrink}, for all $r\ge 0$ and $\lambda\in[0,1]$ we get
\[
 f_{\tau,\mu}((1-\lambda)r) \le \left(1-\frac{\lambda}{2}\right)\psi_{\tau}(r)+\mu(1-\lambda)^2r^2.
\]
Since
\[
 \mu(1-\lambda)^2 = \mu\left(1-\frac{\lambda}{2}\right)-\lambda\left(1-\frac{\lambda}{2-\kappa}\right)\frac{\kappa\gamma}{2},
\]
with $\kappa=1/2$ and $\gamma=6\mu$, it follows that
\[
 f_{\tau,\mu}((1-\lambda)r) \le \left(1-\frac{\lambda}{2}\right)f_{\tau,\mu}(r)
 -\lambda\left(1-\frac{\lambda}{2-\kappa}\right)\frac{\kappa\gamma}{2}\,r^2.
\]
Due to $f_{\tau,\mu}(0)=0$, this is exactly the $(\kappa,\gamma)$-strong quasar-convexity of $f_{\tau,\mu}$ with respect to $\ov{z}=0$.

This scalar inequality directly produces the strong quasar-convexity of $H_{\tau,\mu,q}$. Indeed, fix $z\neq0$ and write $z=ru$ with $r=\norm z$ and $u\in\mathbb S^{d-1}$. Since the segment from $0$ to $z$ maintains the same direction $u$, multiplying the scalar inequality by $q(u)$ gives
\[
 H_{\tau,\mu,q}((1-\lambda)z) \le \left(1-\frac{\lambda}{2}\right)H_{\tau,\mu,q}(z)
 -\lambda\left(1-\frac{\lambda}{2-\kappa}\right)\frac{\kappa\gamma}{2}\,q(u)\norm z^2.
\]
Together with $q(u)\ge1$ and $H_{\tau,\mu,q}(0)=0$, this implies
\[
 H_{\tau,\mu,q}((1-\lambda)z) \le \kappa\lambda H_{\tau,\mu,q}(0)+(1-\kappa\lambda)H_{\tau,\mu,q}(z)
 -\lambda\left(1-\frac{\lambda}{2-\kappa}\right)\frac{\kappa\gamma}{2}\norm z^2,
\]
which is exactly $(\kappa,\gamma)$-strong quasar-convexity with respect to $\ov z=0$.

Next, we prove the local Lipschitz continuity of $H_{\tau,\mu,q}$. Since $q$ is $C^1$ on the compact sphere, it is Lipschitz and bounded there. Hence, there exist constants $L_q, M_q>0$ such that
\[
 |q(u_1)-q(u_2)|\le L_q\norm{u_1-u_2}, \qquad
 1\le q(u)\le M_q, \qquad \forall u,u_1,u_2\in\mathbb S^{d-1}.
\]
Let us write
\[
 H_{\tau,\mu,q}(z)=F_1(z)+F_2(z),
\]
where
\[
 F_1(z):= \begin{cases}
  \psi_\tau(\norm z)\,q\!\left(\dfrac{z}{\norm z}\right), & z\neq 0,\\[1mm]
  0,& z=0,
 \end{cases}
 \qquad
 F_2(z):= \begin{cases}
  \mu\norm z^2\,q\!\left(\dfrac{z}{\norm z}\right), & z\neq 0,\\[1mm]
  0,& z=0.
 \end{cases}
\]
Let $z_1,z_2\neq 0$, set $r_i:=\norm{z_i}$ and $u_i:=z_i/r_i$, and assume without loss of generality that $r_1\ge r_2$. Since $\psi_\tau$ is $1$-Lipschitz and $\psi_\tau(r_2)\le r_2$, we have
\begin{align*}
    |F_1(z_1)-F_1(z_2)|&\le |\psi_\tau(r_1)-\psi_\tau(r_2)|\,|q(u_1)|+\psi_\tau(r_2)|q(u_1)-q(u_2)|\\
 &\le M_q|r_1-r_2|+r_2L_q\norm{u_1-u_2}. 
\end{align*}
In addition, it holds that
\begin{align*}
 r_2\norm{u_1-u_2} =\norm{r_2u_1-r_2u_2} &\le \norm{r_2u_1-r_1u_1}+\norm{r_1u_1-r_2u_2}
 \\&= |r_1-r_2|+\norm{z_1-z_2} \le 2\norm{z_1-z_2},
\end{align*}
i.e.,
\[
 |F_1(z_1)-F_1(z_2)|\le (M_q+2L_q)\norm{z_1-z_2}.
\]
If one of the points is $0$, the same bound follows from $\psi_\tau(r)\le r$ and $q\le M_q$, i.e., $F_1$ is globally Lipschitz.

Now, let us fix $R>0$ and let $z_1,z_2\in\mb(0,R)\setminus\{0\}$. Together with $r_1+r_2\le 2R$, the estimate above for $r_2\norm{u_1-u_2}$ leads to
\[
 \begin{aligned}
 |F_2(z_1)-F_2(z_2)| &\le \mu M_q|r_1^2-r_2^2|+\mu r_2^2L_q\norm{u_1-u_2}\\
 &\le 2\mu M_qR|r_1-r_2|+\mu r_2L_q\left(r_2\norm{u_1-u_2}\right)\\
 &\le 2\mu(M_q+L_q)R\norm{z_1-z_2}.
 \end{aligned}
\]
If one point is zero, say $z_2=0$, then for $z_1\in \mb(0,R)$,
\[
 |F_2(z_1)-F_2(0)| =\mu\norm{z_1}^2 q(z_1/\norm{z_1}) \le \mu M_q R\norm{z_1} = \mu M_q R\norm{z_1-z_2}.
\]
As a consequence, $F_2$ is Lipschitz on all bounded balls, i.e., locally Lipschitz on $\R^d$. Since $H_{\tau,\mu,q}=F_1+F_2$, the map $H_{\tau,\mu,q}$ is locally Lipschitz on $\R^d$. Moreover, the function is nonsmooth because along any ray $z=ru$ the derivative jumps at $r=\tau$.

Finally, we show that $H_{\tau,\mu,q}$ is not star-convex with respect to $\ov z =0$. Choose any unit vector $u$ and set $z:=4\tau u$. Then $\norm z>\tau$ and $\norm{z/4}=\tau$, so
\[
 H_{\tau,\mu,q}(z/4)-\frac14 H_{\tau,\mu,q}(z)
 = q(u)\left(\tau+\mu\tau^2-\frac14\left(2\tau+16\mu\tau^2\right)\right)
 = q(u)\tau\left(\frac12-3\mu\tau\right)>0,
\]
thanks to $\mu\tau<1/8<1/6$. Hence, $H_{\tau,\mu,q}$ is not star-convex with the center $\ov z =0$.
\end{proof}

\begin{example}[\textbf{Anisotropic robust model stitching}]
\label{ex:anisotropic-model-stitching}
Model stitching studies whether internal representations of two frozen networks can be connected by a low-capacity map, often a linear or affine
transformation~\cite{Csiszarik2021RepresentationMatching}. Let $X=[x_1,\ldots,x_N]\in\mathbb R^{d_x\times N}$ and
$Y=[y_1,\ldots,y_N]\in\mathbb R^{d_y\times N}$ denote paired hidden representations from two frozen models. Consider the affine connector
$x\mapsto Ax+b$. With 
$\widetilde X :=\begin{bmatrix}X\\\mathbf 1^\top\end{bmatrix}$ and $\widetilde A:=[A\ b]$, define the anisotropic robust stitching loss
\[
h_{\rm stitch}(\widetilde A):=H_{\tau,\mu,q}\left(\operatorname{vec}(\widetilde A\widetilde X-Y)\right),
\]
where now $q:\mathbb S^{d_yN-1}\to[1,\infty)$ is a nonconstant $C^1$ angular weight.
The angular factor $q$ allows direction-dependent penalties in the hidden representation mismatch, while the kinked radial profile gives a nonsmooth
nonconvex discrepancy.

Assume that $\widetilde X$ has full row rank and that there exists $\widetilde A^\star$ such that $\widetilde A^\star \widetilde X=Y$.
Let $\sigma_{\min}(\widetilde X)$ denote the smallest singular value of $\widetilde X$. 
Then $h_{\rm stitch}$ is locally Lipschitz, non-star-convex, and
$\left(\frac12,6\mu\,\sigma_{\min}(\widetilde X)^2\right)$-strongly quasar-convex with respect to $\widetilde A^\star$.
\end{example}
\begin{proof}
Let $\mathcal L:\mathbb R^{d_y\times (d_x+1)}\to \mathbb R^{d_yN}$ be defined as $\mathcal L(U):=\operatorname{vec}(U\widetilde X)$.
It follows from $\widetilde A^\star\widetilde X=Y$ that,
for every $\widetilde A\in\mathbb R^{d_y\times(d_x+1)}$,  
\[
\operatorname{vec}(\widetilde A\widetilde X-Y)=\operatorname{vec}\bigl((\widetilde A-\widetilde A^\star)\widetilde X\bigr)
=\mathcal L(\widetilde A-\widetilde A^\star),
\]
i.e.,
\[
h_{\rm stitch}(\widetilde A)=H_{\tau,\mu,q}\bigl(\mathcal L(\widetilde A-\widetilde A^\star)\bigr).
\]
We first record the rank property of $\mathcal L$. Under the standard
vectorization identity,
\[
\mathcal L(U)=(\widetilde X^\top\otimes I_{d_y})\operatorname{vec}(U).
\]
Since $\widetilde X$ has the full row rank, the matrix
$\widetilde X^\top$ has the full column rank, i.e., $\widetilde X^\top\otimes I_{d_y}$ has the full column rank, and
\[
\sigma_{\min}(\mathcal L)=\sigma_{\min}(\widetilde X^\top\otimes I_{d_y})=\sigma_{\min}(\widetilde X).
\]
In particular, $\mathcal L$ is injective. 

The local Lipschitz continuity of $h_{\rm stitch}$ follows from Proposition~\ref{prop:master-family} because $H_{\tau,\mu,q}$ is locally Lipschitz and $h_{\rm stitch}$ is its composition with an affine map.

Next, we show that $h_{\rm stitch}$ is nonsmooth. Since $\mathcal L$ is injective, for every nonzero $U\in\mathbb R^{d_y\times(d_x+1)}$, we have
$\mathcal L U\neq0$. Fix such a $U$, and set $v:=\mathcal{L} U$ and $u:=\frac{v}{\|v\|}$.
Continuing along the lines $\widetilde A^\star+tU$, we have
\[
h_{\rm stitch}(\widetilde A^\star+tU) = H_{\tau,\mu,q}(tv).
\]
For sufficiently small $t>0$,
\[
H_{\tau,\mu,q}(tv)=\bigl(t\|v\|+\mu t^2\|v\|^2\bigr)q(u),
\]
whereas for sufficiently small $t<0$,
\[
H_{\tau,\mu,q}(tv) =\bigl(|t|\|v\|+\mu t^2\|v\|^2\bigr)q(-u).
\]
Consequently, the one-dimensional restriction $t\mapsto h_{\rm stitch}(\widetilde A^\star+tU)$ has a kink at $t=0$. Hence $h_{\rm stitch}$ is nonsmooth at
$\widetilde A^\star$.

We now prove that $h_{\rm stitch}$ is not star-convex with respect to $\widetilde A^\star$. Since $\mathcal L$ is injective, choose any
$U\neq0$, and set $u:=\frac{\mathcal L U}{\|\mathcal L U\|}$ and $t:=\frac{4\tau}{\|\mathcal L U\|}$.
Then $t\mathcal L U=4\tau u$. By Proposition~\ref{prop:master-family}, the function $H_{\tau,\mu,q}$ is not star-convex with respect to $0$. More explicitly, using the computation in the proof of Proposition~\ref{prop:master-family},
\[
H_{\tau,\mu,q}(\tau u)> \frac14 H_{\tau,\mu,q}(4\tau u),
\]
i.e.,
\[
 H_{\tau,\mu,q}\left(\frac{t}{4}\mathcal L U\right)>\frac14 H_{\tau,\mu,q}(t\mathcal L U).
\]
Using the identity $h_{\rm stitch}(\widetilde A^\star+sU)=H_{\tau,\mu,q}(s\mathcal L U)$, we come to
\[
 h_{\rm stitch}\left(\widetilde A^\star+\frac{t}{4}U\right)>\frac14 h_{\rm stitch}(\widetilde A^\star+tU).
\]
In addition, it can be deduced that
\[
\widetilde A^\star+\frac{t}{4}U =\frac34\widetilde A^\star+\frac14(\widetilde A^\star+tU).
\]
Since $h_{\rm stitch}(\widetilde A^\star)=0$, the star-convexity with respect to $\widetilde A^\star$ implies
\[
 h_{\rm stitch}\left(\frac34\widetilde A^\star+\frac14(\widetilde A^\star+tU)\right)\le\frac34 h_{\rm stitch}(\widetilde A^\star)+
\frac14 h_{\rm stitch}(\widetilde A^\star+tU)=\frac14 h_{\rm stitch}(\widetilde A^\star+tU),
\]
which contradicts the strict inequality above. Therefore, $h_{\rm stitch}$ is not star-convex with respect to $\widetilde A^\star$.

It remains to prove strong quasar-convexity.
Since $H_{\tau,\mu,q}$ is $\left(\frac12,6\mu\right)$-strongly quasar-convex with respect to $0\in\argmin{} H_{\tau,\mu,q}$ and
since $\mathcal L$ has full column rank and $\mathcal L(0)=0\in \argmin{} H_{\tau,\mu,q}$,
the linear-composition rule for strongly quasar-convex functions (see \cite[Proposition~21]{Ahookhosh2026Quasar}) gives that the function $U\mapsto H_{\tau,\mu,q}(\mathcal L U)$ is $\left(\frac12,6\mu\,\sigma_{\min}(\mathcal L)^2\right)=\left(\frac12,6\mu\,\sigma_{\min}(\widetilde X)^2\right)$-strongly quasar-convex with respect to $U=0$. Finally, the translation $U=\widetilde A-\widetilde A^\star$ preserves strong quasar-convexity, and
therefore $h_{\rm stitch}$ is $\left(\frac12,6\mu\,\sigma_{\min}(\widetilde X)^2\right)$-strongly quasar-convex with respect to $\widetilde A^\star$.
\end{proof}

\section{Inexact high-order proximal-point algorithm} \label{sec:inexatHiPPA}
We now turn to the algorithmic framework. The goal is to solve
\begin{equation}\label{eq:mainproblem}
    \mathop{\bs\min}\limits_{x\in \R^n}\ h(x),
\end{equation}
under the following standing assumptions:
\begin{assumption}[\textbf{Basic assumptions}]\label{ass:basic} 
The function $h:\R^n\to\Rinf$ is proper and lower semicontinuous, $\dom h$ is convex, and $X^\star\ne\emptyset$.  Unless otherwise stated, $h$ is $\kappa$-quasar-convex or $(\kappa,\gamma)$-strongly quasar-convex with respect to a fixed minimizer $\xb\in X^\star$.  The proximal parameters satisfy
\begin{equation}\label{eq:beta-bounds}
        0<\beta_{\min}\le \beta_k\le \beta_{\max}<\infty.
\end{equation}
\end{assumption}

For $x\in\R^n$, $\beta>0$, and $p>1$, let us define
\begin{equation}\label{eq:Qxbeta}
Q_{x,\beta}(y):=h(y)+\frac{1}{p\beta}\norm{y-x}^{p},
\qquad T_{\beta,p}(x):=\argmint{y\in\R^n}Q_{x,\beta}(y).
\end{equation}
At the iteration $k$, we write $Q_k:=Q_{x^k,\beta_k}$.
Since $h$ is bounded below by $h^\star$, the function
$Q_{x,\beta}$  is proper, lower semicontinuous, bounded below, and coercive. Hence $\inf_{\R^n} Q_{x,\beta}$ is attained, and $T_{\beta,p}(x)\neq\emptyset$. 
In the current solution $x^k$, for a given $\beta_k>0$ and for $p>1$, the exact proximal oracle in the high-order proximal-point method (HiPPA) computes a minimizer of $Q_k$ and set $x^{k+1}\in T_{\beta_k,p}(x^k)$. 
Since this subproblem is generally difficult to solve exactly for nonsmooth and nonconvex objectives, we next introduce inexact proximal steps.

\subsection{Model-value gap, residual, and metric inexactness}

The exact high-order proximal-point method computes a minimizer of the regularized model $Q_k$ at each iteration.  For robust nonsmooth objectives, this subproblem is rarely solved exactly. 
We therefore introduce three inexactness criteria for approximate proximal steps.
The first measures suboptimality in the proximal model value, the second measures approximate stationarity of the proximal model, and the third measures distance to the exact proximal solution set.

\begin{definition}[\textbf{Inexact proximal step}]\label{def:inexact_defs}
Given $x^k\in\dom h$, $\beta_k>0$, and $p>1$, a point $x^{k+1}\in\dom h$ is called
\begin{enumerate}[label=(\textbf{\alph*}), font=\normalfont\bfseries, leftmargin=0.7cm]
\item \label{def:value-gap} a \textit{Model-value gap inexact proximal step} with accuracy $\eps_k\ge0$ if
\begin{equation}\label{eq:value-gap}
        Q_k(x^{k+1})\le \inft{y\in\R^n}Q_k(y)+\eps_k.
\end{equation}
\item \label{def:residual} a \textit{residual inexact proximal step} with tolerance $\eta_k\ge0$ if 
if $h$ is locally Lipschitz around $x^{k+1}$ and there exists $r^{k+1}\in\R^n$ such that
\begin{equation}\label{eq:residual-inclusion}
r^{k+1}\in \partial^C h(x^{k+1})+\frac1{\beta_k}\Jp(x^{k+1}-x^k),\qquad \text{with}\qquad \norm{r^{k+1}}\le \eta_k.
\end{equation}
\item \label{def:metric} a \textit{metric inexact proximal step} with accuracy $\delta_k\ge0$ if
\begin{equation}\label{eq:metric-inexact}
\dist\big(x^{k+1},T_{\beta_k,p}(x^k)\big)\le \delta_k .
\end{equation}
\end{enumerate}
\end{definition}

These notions serve different purposes.
The model-value gap condition is well-suited for descent estimates because it compares the computed point
directly with the minimum value of the proximal model. However, certifying this condition requires either the exact value $\inf Q_k$ or a valid lower bound on it.
Many first-order and quasi-Newton solvers instead provide residual or progress-based stopping rules.
Residual inexactness is the natural condition for stationarity, since it approximates the optimality condition of the proximal model. 
Metric inexactness is useful for transferring local rates of the exact proximal map to inexact iterations. 
These notions are not equivalent, in general. In particular, for nonconvex proximal subproblems, a small model-value gap does not automatically imply a small residual unless an additional error-bound or growth condition is available.

The following remark clarifies the well-definedness of Definition~\ref{def:inexact_defs}~\ref{def:value-gap}.
\begin{remark}
Since $Q_k$ is proper, lower semicontinuous, bounded below, and coercive, its infimum is attained. 
Therefore, for every $\eps_k\ge0$, the model-value gap condition
\[
Q_k(x^{k+1})\le \inf_{y\in\R^n}Q_k(y)+\eps_k
\]
is well defined. In particular, when $\eps_k=0$, this condition reduces to
the exact high-order proximal step $x^{k+1}\in T_{\beta_k,p}(x^k)$.
Thus, the model-value gap oracle is a direct relaxation of the exact proximal oracle used in HiPPA \cite{Ahookhosh2026Quasar}.
\end{remark}

By Definition~\ref{def:inexact_defs}~\ref{def:value-gap}, for every $z\in\R^n$, it can be concluded that
\begin{equation}\label{eq:upper-model}
 Q_k(x^{k+1})\le Q_k(z)+\varepsilon_k.
\end{equation}
In particular, choosing $z=x^k$ yields the descent estimate
\begin{equation}\label{eq:descent-gap}
 h(x^{k+1})+\frac{1}{p\beta_k}\norm{x^{k+1}-x^k}^p\le h(x^k)+\varepsilon_k.
\end{equation}
Setting $z=\xb$ leads to
\begin{equation}\label{eq:compare-opt}
 h(x^{k+1})-h^\star\le \frac{1}{p\beta_k}\left(\norm{x^k-\xb}^p-\norm{x^{k+1}-x^k}^p\right)+\varepsilon_k.
\end{equation}
\begin{remark}[\textbf{A model-value gap does not automatically imply a residual bound}]\label{rem:value-gap-residual-warning}
For nonsmooth proximal models, a small model-value gap does not generally imply a small Clarke residual.  For example, for $Q(y)=|y|$, the value-based gap $Q(y)-\inf Q$ tends to zero as $y\to0$, but $\dist(0,\partial^C Q(y))=1$ for $y\ne0$. Thus, converting a model-value gap estimate into a residual estimate requires an
additional error-bound or growth condition for the proximal model.
\end{remark}

The model-value gap condition gives descent estimates, but it does not by itself imply approximate stationarity for nonsmooth proximal subproblems. We therefore
use the following local error-bound assumption to convert model suboptimality into a residual bound.
\begin{assumption}[\textbf{Local residual error bound for the proximal model}]\label{ass:residual-eb}
Let $B\subset\R^n$ be compact and let $I := [\beta_{\min},\beta_{\max}]$.  Assume that $h$ is finite and locally Lipschitz on an open neighborhood of $B$. Suppose that there exists $L_{B,I}>0$ such that, for all $x,y\in B$ and all $\beta\in I$,
\begin{equation}\label{eq:residual-eb}
\dist\left(0,\partial^C h(y)+\frac1\beta\Jp(y-x)\right)\le L_{B,I}\sqrt{Q_{x,\beta}(y)-\inft{y\in\R^n} Q_{x,\beta}(y)}.
\end{equation}
\end{assumption}

The next proposition shows that, under this local error bound, every model-value gap inexact step automatically satisfies a residual inexactness condition.
\begin{proposition}[\textbf{Model-value gap implies residual inexactness}]\label{prop:value-to-residual}
Suppose Assumption~\ref{ass:residual-eb} holds on $(B,I)$.  If $x^k,x^{k+1}\in B$, $\beta_k\in I$, and $x^{k+1}$ satisfies \eqref{eq:value-gap}, then $x^{k+1}$ satisfies the residual condition \eqref{eq:residual-inclusion} with
$\eta_k=L_{B,I}\sqrt{\eps_k}$.
\end{proposition}
\begin{proof}
Since $x^k,x^{k+1}\in B$ and $\beta_k\in I$, applying \eqref{eq:residual-eb} to $x=x^k$, $\beta=\beta_k$, and $y=x^{k+1}$ gives
\[
 \dist\left(0,\partial^C h(x^{k+1})+\tfrac1{\beta_k}\Jp(x^{k+1}-x^k)\right) \le L_{B,I}\sqrt{Q_k(x^{k+1})-\inf_{z\in\R^n} Q_k(z)}.
\]
By the model-value gap condition \eqref{eq:value-gap}, the right-hand side is at most $L_{B,I}\sqrt{\eps_k}$.
Since $h$ is locally Lipschitz around $x^{k+1}$, the Clarke subdifferential $\partial^C h(x^{k+1})$ is nonempty and compact. Hence, the shifted set
$\partial^C h(x^{k+1}) +\frac1{\beta_k}\Jp(x^{k+1}-x^k)$ is also nonempty and compact, so the distance to this set is attained. Therefore,
there exists $ r^{k+1}\in \partial^C h(x^{k+1})+\frac1{\beta_k}\Jp(x^{k+1}-x^k)$
such that $\norm{r^{k+1}}\le L_{B,I}\sqrt{\eps_k}$. Thus, $x^{k+1}$ satisfies the residual inexactness condition
\eqref{eq:residual-inclusion} with $\eta_k=L_{B,I}\sqrt{\eps_k}$.
\end{proof}

Proposition~\ref{prop:value-to-residual} shows that a model-value gap estimate can be converted to residual inexactness under a local residual error bound. We now give an analogous bridge from model-value gap inexactness to metric inexactness. This will be useful later when transferring local rates of the exact
proximal map to inexact proximal iterations.

\begin{assumption}[\textbf{Local metric growth of the proximal model}]\label{ass:metric-growth}
There exist neighborhoods $U,V$ of $\xb$, constants $c_{\rm mg}>0$ and $s\ge1$, and an interval $I=[\beta_{\min},\beta_{\max}]$, such that, for all
$x\in U$, all $y\in V$, and all $\beta\in I$,
\begin{equation}\label{eq:metric-growth}
Q_{x,\beta}(y)-\inf_{z\in\R^n} Q_{x,\beta}(z)\ge c_{\rm mg}\,\dist\big(y,T_{\beta,p}(x)\big)^s.
\end{equation}
\end{assumption}

\begin{proposition}[\textbf{Model-value gap implies metric inexactness}]\label{prop:value-to-metric}
Suppose Assumption~\ref{ass:metric-growth} holds.  If $x^k\in U$, $x^{k+1}\in V$, $\beta_k\in I$, and $x^{k+1}$ satisfies \eqref{eq:value-gap}, then
\begin{equation}\label{eq:value-to-metric-bound}
        \dist\big(x^{k+1},T_{\beta_k,p}(x^k)\big)\le c_{\rm mg}^{-1/s}\eps_k^{1/s}.
\end{equation}
\end{proposition}
\begin{proof}
By Assumption~\ref{ass:metric-growth}, applied with $x=x^k$, $y=x^{k+1}$, and $\beta=\beta_k$, we have
\[
c_{\rm mg}\dist(x^{k+1},T_{\beta_k,p}(x^k))^s\le Q_k(x^{k+1})-\inf Q_k\le \eps_k.
\]
Taking the $s$-th root gives our desired result.
\end{proof}

\subsection{Inexact high-order proximal-point algorithm}
We now state the inexact high-order proximal-point method (HiPPA) in Algorithm~\ref{alg:ihippa}. The algorithm is written using a model-value gap oracle. In implementations, this oracle can be replaced by a computable certificate, such as a certified dual gap, a valid lower bound on the proximal model, or a residual stopping rule when an error bound for the proximal model is available. The convergence analysis below makes explicit which type of
inexactness is needed for each conclusion.

\begin{algorithm}[h]
\caption{HiPPA (Inexact High-Order Proximal-Point Algorithm)}
\label{alg:ihippa}
\begin{algorithmic}[1]
\State \textbf{Input:}  initial point $x^0\in\dom h$, order $p>1$, bounds $0<\beta_{\min}\le \beta_{\max}<\infty$, tolerances 
$\{\varepsilon_k\}_{k\in\Nz}$.
\While{stopping criteria are not satisfied}
    \State Choose $\beta_k\in[\beta_{\min},\beta_{\max}]$.
    \State Compute $x^{k+1}\in\dom h$ such that
    \[
        Q_{x^k,\beta_k}(x^{k+1})\le \inf_y Q_{x^k,\beta_k}(y)+\eps_k.
    \]
\EndWhile
\end{algorithmic}
\end{algorithm}

Before proving convergence, we collect three elementary estimates that will be used repeatedly. The first one follows directly from the model value-gap condition and gives the basic descent mechanism of the method.
\begin{lemma}[\textbf{Model-value gap descent inequality}]\label{lem:value-gap}
Let $\xb\in X^\star$, and let $\{x^k\}_{k\in\Nz}$ be generated by Algorithm~\ref{alg:ihippa}. Then, for every $k\ge 0$,
\begin{equation}\label{eq:value-gap-main}
 h(x^{k+1})-h(\xb)\le \frac{1}{p\beta_k}\left(\norm{x^k-\xb}^p-\norm{x^{k+1}-x^k}^p\right)+\varepsilon_k.
\end{equation}
Moreover, Moreover, for every $N\ge1$,
\begin{equation}\label{eq:value-gap-sum}
 \sum_{k=0}^{N-1}\frac{1}{p\beta_k}\norm{x^{k+1}-x^k}^p \le h(x^0)-h(\xb)+\sum_{k=0}^{N-1}\varepsilon_k.
\end{equation}
\end{lemma}
\begin{proof}
The inequality \eqref{eq:value-gap-main} is exactly \eqref{eq:compare-opt}. Summing \eqref{eq:descent-gap} from $k=0$ to $N-1$ yields \eqref{eq:value-gap-sum}.
\end{proof}
The next estimate is a perturbed Fej\'{e}r-type inequality. It is obtained by applying the residual-to-gap inequality to the approximate optimality condition of the proximal subproblem. The signed form keeps the useful geometry of the
proximal displacement, while the norm form gives a simpler bound.
\begin{lemma}[\textbf{Perturbed one-step inequality}]\label{lem:master}
Assume that $h$ is locally Lipschitz and $(\kappa,\gamma)$-strongly quasar-convex with respect to $\xb\in X^\star$. Let $x^k,x^{k+1}\in\dom h$, and suppose that $x^{k+1}$ admits a residual $r^{k+1}$ satisfying \eqref{eq:residual-inclusion}. Define $d^k:=x^{k+1}-x^k$. Then, for every $k\in\Nz$,
\begin{align}\label{eq:master-signed}
&\kappa\left(h(x^{k+1})-h(\xb)\right)+\frac{\kappa\gamma}{2}\norm{x^{k+1}-\xb}^2
\notag\\&~~~~~~~~~~~~~~~\le \frac{1}{\beta_k}\norm{d^k}^{p-2}\ip{d^k}{\xb-x^{k+1}}+\norm{r^{k+1}}\,\norm{x^{k+1}-\xb}.
\end{align}
Moreover,
\begin{align}\label{eq:master-basic}
&\kappa\left(h(x^{k+1})-h(\xb)\right)+\frac{\kappa\gamma}{2}\norm{x^{k+1}-\xb}^2
\notag\\&~~~~~~~~~~~~~~~\le \frac{1}{\beta_k}\norm{d^k}^{p-1}\norm{x^{k+1}-\xb}+\norm{r^{k+1}}\,\norm{x^{k+1}-\xb}.
\end{align}
If $p=2$, then the sharper identity,
\begin{align}\label{eq:master-p2}
&2\beta_k\kappa\left(h(x^{k+1})-h(\xb)\right)+\beta_k\kappa\gamma\norm{x^{k+1}-\xb}^2
\notag\\&~~~~~~~~~~~~~~~\le \norm{x^k-\xb}^2-\norm{x^{k+1}-\xb}^2-\norm{d^k}^2+2\beta_k\norm{r^{k+1}}\,\norm{x^{k+1}-\xb}.
\end{align}
holds.
\end{lemma}
\begin{proof}
Applying Proposition~\ref{prop:stability-approx-prox} with $x=x^k$ and $y=x^{k+1}$ yields
\[
\kappa\left(h(x^{k+1})-h(\xb)\right)+\frac{\kappa\gamma}{2}\norm{x^{k+1}-\xb}^2\le \frac{1}{\beta_k}\norm{d^k}^{p-2}\ip{d^k}{\xb-x^{k+1}}+\norm{r^{k+1}}\,\norm{x^{k+1}-\xb},
\]
which gives \eqref{eq:master-signed}. Estimate the transport term by Cauchy--Schwarz to obtain \eqref{eq:master-basic}. If $p=2$, then
\[
2\ip{d^k}{\xb-x^{k+1}}=\norm{x^k-\xb}^2-\norm{x^{k+1}-\xb}^2-\norm{d^k}^2,
\]
and multiplying \eqref{eq:master-signed} by $2\beta_k$ yields \eqref{eq:master-p2}.
\end{proof}

The previous lemma is stated under residual inexactness. Combining it with Proposition~\ref{prop:value-to-residual} gives the corresponding version for the model-value gap steps.

\begin{corollary}[\textbf{Model-value gap version of the perturbed one-step inequality}]\label{cor:value-gap-fejer}
Assume that $h$ is locally Lipschitz and $(\kappa,\gamma)$-strongly quasar-convex with respect to $\xb\in X^\star$. Let $x^k,x^{k+1}\in B$, where $B$ is compact, and let $\beta_k\in I\subset(0,\infty)$, where $I$ is compact.
Suppose Assumption~\ref{ass:residual-eb} holds on $(B,I)$, and suppose that $x^{k+1}$ satisfies the model-value gap condition \eqref{eq:value-gap}. Then there exists
\[
 r^{k+1}\in  \partial^C h(x^{k+1}) + \frac1{\beta_k}\Jp(x^{k+1}-x^k)
\]
such that
\[
\norm{r^{k+1}}\le L_{B,I}\sqrt{\eps_k}.
\]
Consequently, all inequalities in Lemma~\ref{lem:master} hold with this residual bound.
\end{corollary}
\begin{proof}
The residual existence and the bound $\norm{r^{k+1}}\le L_{B,I}\sqrt{\eps_k}$ follow from Proposition~\ref{prop:value-to-residual}. Applying Lemma~\ref{lem:master} with this residual gives the stated inequalities.
\end{proof}


We now analyze Algorithm~\ref{alg:ihippa}. The first result uses only the
model-value gap condition, the summability of the errors, and the uniform upper bound on the proximal parameters 
$\{\beta_k\}_{k\in\Nz}$.

\begin{theorem}[\textbf{Descent and step summability}]\label{thm:global-quasi}
Let Assumption~\ref{ass:basic} hold, and let $\{x^k\}_{k\in\Nz}$ be generated by Algorithm~\ref{alg:ihippa}.  If
\begin{equation}\label{eq:summable-forcing}
 \sum_{k=0}^\infty \varepsilon_k<\infty.
\end{equation} 
Then, the following statements hold:
\begin{enumerate}[label=(\textbf{\alph*}), font=\normalfont\bfseries, leftmargin=0.7cm]
    \item\label{thm:global-quasi:a} For each $k\in\Nz$, $Q_k(x^{k+1})\le Q_k(x^{k})+\varepsilon_k$ and 
    $h(x^{k+1})\leq h(x^{k})+\varepsilon_k$;
    \item\label{thm:global-quasi:b} We have $\sum_{k=0}^\infty \norm{x^{k+1}-x^k}^p<\infty$ and in particular, $\norm{x^{k+1}-x^k}\to 0$.
\end{enumerate}
\end{theorem}
\begin{proof}
Assertion~\ref{thm:global-quasi:a} follows from the descent estimate \eqref{eq:descent-gap}. For Assertion~\ref{thm:global-quasi:b}, Lemma~\ref{lem:value-gap} implies, for
every $N\ge1$,
\[
 \sum_{k=0}^{N-1} \frac{1}{p\beta_k}\norm{x^{k+1}-x^k}^p \le h(x^0)-h(\xb)+\sum_{k=0}^{N-1}\eps_k .
\]
Since $\xb\in X^\star$, $h(\xb)=h^\star$, and since $\sum_{k=0}^{\infty}\eps_k<\infty$, the right-hand side is uniformly bounded
in $N$. In light of $\beta_k\le \beta_{\max}$, it can be concluded that
\[
 \frac{1}{p\beta_{\max}} \sum_{k=0}^{N-1}\norm{x^{k+1}-x^k}^p \le h(x^0)-h^\star+\sum_{k=0}^{\infty}\eps_k .
\]
Letting $N\to\infty$ yields
\[
\sum_{k=0}^{\infty}\norm{x^{k+1}-x^k}^p<\infty,
\]
and consequently we come to $\norm{x^{k+1}-x^k}\to0$, giving our desired result.
\end{proof}
The next result identifies the cluster points of the inexact HiPPA sequence. The main mechanism is residual stationarity: once the steps vanish and the residuals go to zero, every cluster point is Clarke critical, and quasar-convexity then implies global optimality. We first state this mechanism under residual inexactness and then apply it to model-value gap steps through Proposition~\ref{prop:value-to-residual}.

We first provide the following auxiliary result.
\begin{lemma}[\textbf{Cluster optimality under residual inexactness}]\label{lem:cluster-opt}
Let $h$ be locally Lipschitz on an open set containing a bounded sequence $\{x^k\}_{k\in\Nz}$, and suppose that $h$ is $\kappa$-quasar-convex with respect to $\xb\in X^\star$.  Assume \eqref{eq:beta-bounds},
that $\norm{x^{k+1}-x^k}\to 0$, and suppose there exist residuals $r^{k+1}$ satisfying \eqref{eq:residual-inclusion} with $\norm{r^{k+1}}\to0$.  Then every cluster point of $\{x^k\}_{k\in\Nz}$ belongs to $X^\star$.  If $X^\star=\{\xb\}$, then $x^k\to\xb$.
\end{lemma}
\begin{proof}
Let $\widehat x$ be a cluster point, and choose a subsequence $x^{k_j}\to\widehat x$.  Since $\norm{x^{k+1}-x^k}\to0$, we also have $x^{k_j+1}\to\widehat x$.  From \eqref{eq:residual-inclusion}, choose $\xi^{k+1}\in\partial^C h(x^{k+1})$ such that
\[
r^{k+1}=\xi^{k+1}+\frac1{\beta_k}\Jp(x^{k+1}-x^k).
\]
It light of $\beta_k\ge\beta_{\min}>0$, it is clear that
\[
\left\|\frac1{\beta_k}\Jp(x^{k+1}-x^k)\right\|\le\frac1{\beta_{\min}}\norm{x^{k+1}-x^k}^{p-1}\to 0.
\]
Since $\norm{r^{k+1}}\to0$, it follows that
$\xi^{k_j+1}\to0$. The outer semicontinuity of $\partial^C h$ gives $0\in\partial^C h(\hat x)$.  Fact~\ref{crit:are:glob} implies $\widehat x\in X^\star$.  If $X^\star=\{\xb\}$, every cluster point is $\xb$, and the boundedness implies $x^k\to\xb$.
\end{proof}

\begin{theorem}[\textbf{Global convergence under model-value gap errors}]\label{thm:cluster-value-gap}
Let Assumption~\ref{ass:basic} hold, and let $\{x^k\}_{k\in\Nz}$ be generated by Algorithm~\ref{alg:ihippa}.
Assume \eqref{eq:summable-forcing} and suppose that the sequence 
$\{x^k\}_{k\in\Nz}$ is bounded.  Let $B$ be a compact set containing the iterates, and assume that Assumption~\ref{ass:residual-eb} holds on $(B, I)$.  Then every cluster point of $\{x^k\}_{k\in\Nz}$ is globally optimal.
Moreover, $h(x^k)\to h^\star$. If $X^\star=\{\xb\}$, then $x^k\to\xb$.
\end{theorem}
\begin{proof}
By Theorem~\ref{thm:global-quasi}, we have $\norm{x^{k+1}-x^k}\to0$.
Since $\sum_{k=0}^\infty\eps_k<\infty$, we also have $\eps_k\to0$.
By Proposition~\ref{prop:value-to-residual}, for every $k$ there exists
\[
 r^{k+1} \in \partial^C h(x^{k+1}) + \frac1{\beta_k}\Jp(x^{k+1}-x^k)
\]
such that
\[
        \norm{r^{k+1}}\le L_{B,I}\sqrt{\eps_k}.
\]
Hence, we come to $\norm{r^{k+1}}\to 0$.
Lemma~\ref{lem:cluster-opt} then implies that every cluster point of $\{x^k\}_{k\in\Nz}$ belongs to $X^\star$. If $X^\star=\{\xb\}$, the same lemma gives $x^k\to\xb$.

It remains to verify the convergence of the function values. Let us define the tail error as
\[
E_k:=\sum_{j=k}^{\infty}\varepsilon_j .
\]
Since $\sum_k\varepsilon_k<+\infty$, we have $E_k<+\infty$ and $E_k\to 0$. In addition, we have
\[
        E_k=\varepsilon_k+E_{k+1}.
\]
It follows from the descent estimate \eqref{eq:descent-gap} that
\[
 h(x^{k+1})+E_{k+1} \le h(x^k)+\varepsilon_k+E_{k+1} = h(x^k)+E_k .
\]
Thus, the sequence $a_k:=h(x^k)+E_k$ is nonincreasing. Since $h(x^k)\ge h^\star$, we also have $a_k\ge h^\star$. Hence $\{a_k\}_{k\in\Nz}$ converges to some finite number $\ell\ge h^\star$. Because $E_k\to 0$, it follows that
\[
        h(x^k)=a_k-E_k\to \ell .
\]
Now let $\hat x$ be any cluster point and take a subsequence $x^{k_j}\to \hat x$. By the first part of the proof, $\widehat x\in X^\star$. Since Assumption~\ref{ass:residual-eb} implies that $h$ is locally Lipschitz on an open neighborhood of $B$, $h$ is continuous
at $\widehat x$. Hence $ h(x^{k_j})\to h(\widehat x)=h^\star$.
Since the whole sequence $h(x^k)$ converges to $\ell$, we conclude that $\ell=h^\star$. Therefore, it can be concluded that $h(x^k)\to h^\star$, giving our desired result.
\end{proof}

We now consider the local rate regime. When $\gamma>0$, strong quasar-convexity implies that the minimizer is unique; we denote it by $\xb$. The next assumption summarizes the local stability properties needed for the inexact high-order proximal map. The constants may depend on the local geometry
of $h$, the order $p$, and the admissible interval for $\beta_k$.

\begin{assumption}[\textbf{Perturbed local recursions}]\label{ass:perturbed}
Set $r_k:=\norm{x^k-\xb}$. Assume that $h$ is $(\kappa,\gamma)$-strongly quasar-convex with respect to $\ov x$, with $\gamma>0$, and that there exists a radius $R>0$ such that the iterates are eventually contained in $\mb(\xb,R)$.  Moreover, assume that
\begin{enumerate}[label=(\textbf{\alph*}), font=\normalfont\bfseries, leftmargin=0.7cm]
\item \label{ass:perturbed:a} For $p\in (1,2]$, there exist constants $\eta\in(0,1)$ and $C>0$ such that, for all sufficiently large $k$,
\begin{equation}\label{eq:perturbed-p<2}
 r_{k+1}^2\le \eta r_k^2 + C\varepsilon_k.
\end{equation}
\item \label{ass:perturbed:b} For $p>2$, there exist constants $a,b>0$ such that, for all sufficiently large $k$, 
\begin{equation}\label{eq:perturbed-p>2}
 r_{k+1}\le a r_k^{p-1}+b\sqrt{\varepsilon_k}.
\end{equation}
\end{enumerate}
\end{assumption}

Assumption~\ref{ass:perturbed} is a local stability condition for the inexact proximal map. It should not be interpreted as a replacement for the convergence analysis.
The next proposition gives a standard way to verify it from local rates
of the exact proximal map together with a metric inexactness estimate.

\begin{proposition}[\textbf{A verification route for the perturbed recursions}]\label{prop:exact-to-inexact}
Let
\[
y^{k+1}\in \argmint{y\in\R^n} \left\{h(y)+\frac{1}{p\beta_k}\norm{y-x^k}^p\right\}
\]
be an exact high-order proximal point, and suppose 
the computed point $x^{k+1}$ satisfy a metric inexactness estimate
\begin{equation}\label{eq:metric-inexactness-lift}
 \norm{x^{k+1}-y^{k+1}}\le \Delta_k .
\end{equation}
Then the following implications hold.
\begin{enumerate}[label=(\textbf{\alph*}), font=\normalfont\bfseries, leftmargin=0.7cm]
\item\label{prop:exact-to-inexact:a} Suppose $p\in(1,2]$ and the exact proximal step satisfies
\[
 \norm{y^{k+1}-\xb}^2\le \eta\norm{x^k-\xb}^2,
\]
for some $\eta\in(0,1)$. If
$\Delta_k^2\le \ov{C} \varepsilon_k$, for some $\ov{C}>0$, then Assumption~\ref{ass:perturbed}~\ref{ass:perturbed:a} holds for any $\eta'\in(\eta,1)$, after possibly increasing the constant $C$.

\item\label{prop:exact-to-inexact:b} 
Suppose $p>2$ and the exact proximal step satisfies
\[
 \norm{y^{k+1}-\xb}\le a\norm{x^k-\xb}^{p-1},
\]
for some $a>0$. If $\Delta_k\le \ov{C}\sqrt{\varepsilon_k}$, for some $\ov{C}>0$, then Assumption~\ref{ass:perturbed}~\ref{ass:perturbed:b} holds with constants depending only on $a$ and $\ov{C}$.
\end{enumerate}
\end{proposition}
\begin{proof}
\ref{prop:exact-to-inexact:a} By \eqref{eq:metric-inexactness-lift} and Young's inequality, for every $\theta>0$, we get
\[
 r_{k+1}^2=\norm{x^{k+1}-\xb}^2
 \le (1+\theta)\norm{y^{k+1}-\xb}^2+(1+\theta^{-1})\Delta_k^2.
\]
Let us choose $\theta>0$ small enough so that $(1+\theta)\eta<1$, set $\eta'=(1+\theta)\eta$, and use $\Delta_k^2\le \ov{C}\varepsilon_k$. This consequently gives Assumption~\ref{ass:perturbed}~\ref{ass:perturbed:a}. 
\\
\ref{prop:exact-to-inexact:b} Invoking the triangle inequality leads to
\[
 r_{k+1}\le \norm{y^{k+1}-\xb}+\Delta_k
 \le a r_k^{p-1}+\ov{C}\sqrt{\varepsilon_k},
\]
which is Assumption~\ref{ass:perturbed}~\ref{ass:perturbed:b}.
\end{proof}

We now derive local rates from the perturbed recursions in Assumption~\ref{ass:perturbed}. The result separates the two regimes: $1<p\le 2$, where a linear rate follows from a perturbed contraction, and $p>2$, where a compatible forcing rule yields superlinear convergence.

\begin{theorem}[\textbf{Convergence rates of inexact HiPPA}]\label{thm:conv-rates}
Let Assumption~\ref{ass:perturbed} hold, and let $x^k\to\xb$. Then, for all sufficiently large $k$, there exist $A, B>0$ such that
\[
\begin{cases}
\norm{x^{k}-\xb}\le A \rho^k,
&p\in(1,2],~~ {\rm if}~\exists \rho\in(\sqrt{\eta},1), E>0,~{\rm s.t.}~ \varepsilon_k\le E\, \rho^{2k},\\[2mm]
\|x^{k+1}-\xb\| \le B \|x^k-\xb\|^{p-1},
& p>2,~~ {\rm if}~\exists M>0, \delta>0,~{\rm s.t.}~ \sqrt{\varepsilon_k}\le M  \norm{x^k-\xb}^{p-1+\delta}.
\end{cases}
\]
Moreover, if $h$ is locally Lipschitz near $\ov x$, then for each $p>1$ there exists $C_h^p>0$ such that
\begin{equation}\label{eq:value-linear-p<2}
 h(x^k)-h^\star\le C_h^p \rho^k,
\end{equation}
for all sufficiently large $k$.
\end{theorem}
\begin{proof}
We prove the two cases separately. Let us set $ r_k:=\|x^k-\xb\|$.

Suppose $p\in(1,2]$. By Assumption~\ref{ass:perturbed}\ref{ass:perturbed:a}, for all sufficiently large $k$,
\[
 r_{k+1}^2\le \eta r_k^2+C\eps_k .
\]
Let $k_0$ be large enough so that this recursion holds for all $k\ge k_0$, and so that
\[
\eps_k\le E\rho^{2k}, \quad k\ge k_0.
\]
Iterating the recursion for $k>k_0$ leads to
\[
\begin{aligned}
r_k^2&\le\eta^{k-k_0}r_{k_0}^2+C\sum_{j=k_0}^{k-1}\eta^{k-1-j}\eps_j  \\
&\le\eta^{k-k_0}r_{k_0}^2+CE\sum_{j=k_0}^{k-1}\eta^{k-1-j}\rho^{2j}.
\end{aligned}
\]
Due to the inequality $\eta<\rho^2$, the first term satisfies
\[
\eta^{k-k_0}r_{k_0}^2 \le\rho^{2(k-k_0)}r_{k_0}^2 = \rho^{-2k_0}r_{k_0}^2\rho^{2k}.
\]
For the convolution term, we write
\[
\sum_{j=k_0}^{k-1}\eta^{k-1-j}\rho^{2j}=\rho^{2k-2}\sum_{j=k_0}^{k-1}\left(\frac{\eta}{\rho^2}\right)^{k-1-j}\le\rho^{2k-2}\frac{1}{1-\eta/\rho^2},
\]
i.e.,
\[
r_k^2\le\left(\rho^{-2k_0}r_{k_0}^2+\frac{CE\rho^{-2}}{1-\eta/\rho^2}\right)\rho^{2k}.
\]
Consequently, there exists $A>0$ such that $r_k\le A\rho^k$ for all sufficiently large $k$.

If $h$ is locally Lipschitz near $\xb$, then there exists $L>0$ such that, for all sufficiently large $k$,
\[
h(x^k)-h^\star \le L\norm{x^k-\xb},
\]
i.e.,
\[
h(x^k)-h^\star\le C^p_h\rho^k
\]
for a suitable constant $C^p_h>0$.

 Suppose $p>2$. By
In light of Assumption~\ref{ass:perturbed}\ref{ass:perturbed:b}, it is clear that
\[
 r_{k+1}\le a r_k^{p-1}+b\sqrt{\eps_k},
\]
for all sufficiently large $k$. Applying the forcing condition $\sqrt{\eps_k}\le M r_k^{p-1+\delta}$ results in
\[
 r_{k+1} \le a r_k^{p-1}+bM r_k^{p-1+\delta}= \left(a+bM r_k^\delta\right)r_k^{p-1}.
\]
Since $x^k\to\xb$, we have $r_k\to0$. Hence, the factor
$a+bM r_k^\delta$ is bounded for all sufficiently large $k$, i.e., there
exists $B>0$ such that
\[
 r_{k+1}\le B r_k^{p-1},
\]
for all sufficiently large $k$. Dividing both sides by $r_k$ results in
\[
\frac{r_{k+1}}{r_k}\le B r_k^{p-2}\to0,
\]
due to $p>2$ and $r_k\to0$, which consequently leads to a superlinear convergence.

Finally, if $h$ is locally Lipschitz near $\xb$, then for all sufficiently
large $k$,
\[
 h(x^k)-h^\star \le C_h^p\norm{x^k-\xb},
\]
after increasing $C_h^p$ if necessary.
\end{proof}

\begin{remark}[\textbf{Sharper value rates under local upper quadratic growth}]
\label{rem:upper-quadratic-value-rates}
If, in addition to the assumptions of Theorem~\ref{thm:conv-rates}, there exists $L_{\rm uq}>0$ such that
\[
h(x)-h^\star\le L_{\rm uq}\norm{x-\xb}^2
\]
locally around $\xb$, then the value estimates in Theorem~\ref{thm:conv-rates} can be sharpened. In the case $p\in(1,2]$, one
gets
\[
h(x^k)-h^\star\le \widetilde C_h\rho^{2k},
\]
for all sufficiently large $k$. In the case $p>2$, one gets
\[
        h(x^k)-h^\star\le L_{\rm uq}\norm{x^k-\xb}^2,
\]
for all sufficiently large $k$.
\end{remark}

\begin{remark}\label{rem:p>2}
For $p\in(1,2]$, Theorem~\ref{thm:conv-rates} gives $R$-linear convergence
of the iterates. For $p>2$, it gives
\[
\frac{\norm{x^{k+1}-\xb}}{\norm{x^k-\xb}}\to0,
\]
and hence the convergence is superlinear. In particular, the superlinear case is
eventually faster than any prescribed linear rate.
\end{remark}

The following complexity bounds count only outer proximal iterations. They do not include the cost of solving the proximal subproblems, which depends on the inner method, the requested inexactness level, and the availability of computable certificates. Thus, these bounds should be interpreted as outer-iteration complexity estimates, not as total oracle or wall-clock complexity bounds.
\begin{theorem}[\textbf{Outer-iteration complexity}]\label{th:complexity}
Define, for $\epsilon>0$,
\[
        N_r(\epsilon):=\min\left\{N\in\Nz:\ \|x^k-\xb\|\le \epsilon \ \text{for all } k\ge N\right\}.
\]
Under the assumptions of Theorem~\ref{thm:conv-rates}, the following outer-iteration complexity bounds hold:
\begin{equation}\label{eq:complexity-summary-distance}
N_r(\epsilon)=
\begin{cases}
\mathcal O \left(\log(\epsilon^{-1})\right),
& 1<p\le 2,\quad \varepsilon_k=\mathcal O(\rho^{2k}),\\[2mm]
\mathcal O \left(\log\log(\epsilon^{-1})\right),
& p>2,\quad
\varepsilon_k=\mathcal O\left(r_k^{2p-2+2\delta}\right),
\end{cases}
\end{equation}
where $\rho\in(0,1)$ and $\delta>0$ are the constants appearing in the
corresponding forcing rules.
If, in addition, $h$ is locally Lipschitz near $\xb$, then the same outer-iteration orders guarantee
\[
       h(x^k)-h(\xb)\le \epsilon .
\]
Equivalently, if
\[
        N_g(\epsilon):=\min\left\{N\in\mathbb N:\ h(x^k)-h(\xb)\le \epsilon\ \text{for all } k\ge N\right\},
\]
then $N_g(\epsilon)$ satisfies the same bounds as $N_r(\epsilon)$, up to modified constants.
\end{theorem}
\begin{proof}
Set $r_k:=\|x^k-\xb\|$ and $g_k:=h(x^k)-h(\xb)$. We consider the two regimes separately.
First, let $1<p\leq2$. By Theorem~\ref{thm:conv-rates}, under the forcing condition $\eps_k=\mathcal O(\rho^{2k})$, with $\rho\in(0,1)$, there exist constants $C_1>0$ and $k_0\in\mathbb N_0$ such that $r_k\le C_1 \rho^k$, for all $ k\ge k_0$. Thus, $r_k\le\epsilon$ is guaranteed whenever $ C_1 \rho^k\le \epsilon$. Since $\rho\in(0,1)$, this is equivalent to
\[
k \ge \frac{\log(C_1/\epsilon)}{\log(1/\rho)},
\]
leading to $N_r(\epsilon) =  \mathcal O\!\left(\log(\epsilon^{-1})\right)$.

Now, let $p>2$. By Theorem~\ref{thm:conv-rates}, under the forcing condition $\varepsilon_k =\mathcal O\left(r_k^{2p-2+2\delta}\right)$,
there exist constants $C>0$ and $k_0\in\mathbb N$ such that
\[
        r_{k+1}\le C r_k^{p-1} \qquad \forall  k\ge k_0.
\]
Set $s:=p-1>1$. Since $r_k\to0$, increasing $k_0$ if necessary, we may assume that
\[
R_0:=C^{1/(s-1)}r_{k_0}<1.
\]
For $j\ge0$, let us define
\[
        R_j:=C^{1/(s-1)}r_{k_0+j},
\]
i.e.,
\[
R_{j+1}=C^{1/(s-1)}r_{k_0+j+1}\le C^{1/(s-1)}C r_{k_0+j}^{s}=\left(C^{1/(s-1)}r_{k_0+j}\right)^s=R_j^s.
\]
By induction, it holds that $R_j\le R_0^{s^j}$, which ensures
\[
        r_{k_0+j}\le C^{-1/(s-1)}R_0^{s^j}.
\]
To guarantee $r_{k_0+j}\le\epsilon$, it is enough that
\[
        C^{-1/(s-1)}R_0^{s^j}\le \epsilon,
\]
equivalently,
\[
        R_0^{s^j}\le C^{1/(s-1)}\epsilon.
\]
Taking logarithms and using $R_0\in(0,1)$, it suffices that
\[
 s^j\ge\frac{\log\left(C^{-1/(s-1)}\epsilon^{-1}\right)}{\log(R_0^{-1})},
\]
i.e.,
\[
 j\ge\frac{\log\log\left(C^{-1/(s-1)}\epsilon^{-1}\right)-\log\log(R_0^{-1})}{\log s}.
\]
Consequently, $N_r(\epsilon)=\mathcal O\!\left(\log\log(\epsilon^{-1})\right)$. This proves the distance-complexity estimates in
\eqref{eq:complexity-summary-distance}.

Now, we move from distance bounds to function-value bounds. 
Assume that $h$ is locally Lipschitz near $\xb$. Then there exist a neighborhood $U$ of $\xb$ and a constant $L_{\rm loc}>0$ such that
\[
        |h(x)-h(y)|\le L_{\rm loc}\|x-y\|,  \qquad x,y\in U.
\]
Since $x^k\to\xb$, for all sufficiently large $k$, both $x^k$ and $\xb$ belong to $U$. Therefore,
taking
$y=\xb$, we get
\[
        g_k=h(x^k)-h(\xb)\le L_{\rm loc}r_k
\]
for all sufficiently large $k$. Hence, $g_k\le\epsilon$ is guaranteed by
requiring
\[
        r_k\le \frac{\epsilon}{L_{\rm loc}},
\]
i.e., $N_g(\epsilon)$ has the same order as $N_r(\epsilon)$, with only a change in constants. This proves the stated function-value complexity under local Lipschitz continuity. This completes the proof.
\end{proof}

\section{Loss-level robustness under replacement corruption}
\label{sec:loss-robustness}

We evaluate whether the stabilized capped residual loss from Proposition~\ref{prop:stabilized-capped-loss} provides a robust feature-alignment model.
To isolate the effect of the loss from the effect of the optimizer, all losses are trained with Adam \cite{KingmaBa2015} using the same initialization, learning rate, number of iterations, corrupted training data, and random seeds. Thus, the only component changed across methods is the loss function.

We consider the synthetic realizable feature-alignment model $Z_t^{\rm clean}=Z_s W_\star^\top$, where $Z_s\in\mathbb R^{N\times d_s}$, $Z_t^{\rm clean}\in\mathbb R^{N\times d_t}$,
and $W_\star\in\mathbb R^{d_t\times d_s}$. To model an unreliable teacher
supervision, we use replacement-type corruption: for a fraction
$\rho_{\rm corr}$ of samples, the correct teacher feature is replaced by a
teacher feature from another sample,
\[
    z_{t,i}^{\rm train}=z_{t,j}^{\rm clean},\qquad j\neq i.
\]
Unlike additive Gaussian noise, replacement corruption produces plausible but mismatched teacher representations. This models unreliable paired supervision in feature-alignment pipelines, a setting related to feature-based distillation \cite{Romero2015FitNets} and noisy-pair representation learning \cite{Jiang2024NoisyPairs,song2022learning}.

For a residual vector $r$, our proposed loss is
\[
    \rho_{\rm SQC}(r)=\min\{\|r\|,\tau\}+\mu\|r\|^2.
\]
In the experiments, we set $\mu=\mu_\tau/\tau$.
The capped term limits the influence of mismatched teacher features, while the quadratic term provides stabilization and preserves the strong quasar-convex geometry proved in Proposition~\ref{prop:stabilized-capped-loss}. We compare the proposed loss with the squared loss
\[
    \rho_{\rm MSE}(r)=\frac12\|r\|^2,
\]
the Huber loss~\cite{Huber1964},
\[
\rho_{\rm Hub}(r)=\begin{cases}
\frac12\|r\|^2, & \|r\|\le \delta,\\
\delta\bigl(\|r\|-\frac12\delta\bigr), & \|r\|>\delta,
\end{cases}
\]
and three smooth robust losses commonly used in robust estimation: Pseudo-Huber (PH), Cauchy, and Welsch. These losses are also included in the adaptive robust-loss
family of Barron \cite{Barron2019AdaptiveRobustLoss}:
\[
    \rho_{\rm PH}(r)    =    \delta^2\left( \sqrt{1+\frac{\|r\|^2}{\delta^2}}-1\right),
\]
\[
    \rho_{\rm Cauchy}(r) = \frac{c^2}{2} \log\left(1+\frac{\|r\|^2}{c^2}\right), \qquad
    \rho_{\rm Welsch}(r) = \frac{c^2}{2} \left( 1-\exp\left(-\frac{\|r\|^2}{c^2}\right)\right).
\]
For each robust loss, the scale-selection rule is fixed across corruption levels. Specifically, the scale is chosen as a prescribed quantile of the initial residual norms, and the quantile values were selected beforehand using $\rho_{\rm corr}=0.2$. 
Since the ground-truth matrix $W_\star$ is known, we report the relative recovery error
\[
    \frac{\|W-W_\star\|_F}{\|W_\star\|_F}
\]
as the primary metric. We also report the clean MSE,
\[
    \frac{1}{2N}\|Z_sW^\top-Z_t^{\rm clean}\|_F^2,
\]
which measures alignment error against the uncorrupted teacher features.

Table~\ref{tab:replacement-corruption} shows that the behavior of the losses changes sharply once replacement corruption is introduced. At zero corruption, the smooth quadratic-type losses  achieve the smallest errors, as expected in a realizable clean setting. However, their recovery error deteriorates rapidly as the
corruption ratio increases. In contrast, the proposed stabilized capped loss remains stable across corruption levels and achieves the best recovery among the displayed methods at $\rho_{\rm corr}=0.2$ and $\rho_{\rm corr}=0.5$.

\begin{table}[h]
\centering
\caption{
Replacement-corruption robustness of feature-alignment losses. All losses are optimized with Adam \cite{KingmaBa2015} using the same initialization, learning rate, number of iterations, corrupted training data, and random seeds. 
The robust scale-selection rules are fixed across all corruption levels; scale values are
computed from prescribed quantiles of the initial residual norms. The table
reports the relative recovery error, shown as mean $\pm$ standard deviation over 20 seeds.
}
\label{tab:replacement-corruption}
\resizebox{\linewidth}{!}{
\begin{tabular}{llccc}
\toprule
Loss & Setting & $\rho_{\rm corr}=0$ & $\rho_{\rm corr}=0.2$ & $\rho_{\rm corr}=0.5$\\
\midrule
SQC
& $\tau$-q $=0.80,\ \mu_\tau=0.2$
& $2.968{\times}10^{-3}\pm4.6{\times}10^{-4}$
& $\mathbf{2.729{\times}10^{-3}\pm2.4{\times}10^{-4}}$
& $\mathbf{3.195{\times}10^{-3}\pm2.7{\times}10^{-4}}$
\\
Welsch
& scale-q $=0.50$
& $2.241{\times}10^{-5}\pm6.1{\times}10^{-6}$
& $3.703{\times}10^{-2}\pm3.2{\times}10^{-4}$
& $1.594{\times}10^{-1}\pm6.9{\times}10^{-4}$
\\
Cauchy
& scale-q $=0.50$
& $\mathbf{2.101{\times}10^{-5}\pm1.8{\times}10^{-7}}$
& $8.626{\times}10^{-2}\pm4.6{\times}10^{-4}$
& $3.146{\times}10^{-1}\pm1.4{\times}10^{-3}$
\\
Pseudo-Huber
& scale-q $=0.50$
& $2.188{\times}10^{-5}\pm3.0{\times}10^{-7}$
& $1.400{\times}10^{-1}\pm8.7{\times}10^{-4}$
& $4.242{\times}10^{-1}\pm1.8{\times}10^{-3}$
\\
Huber
& scale-q $=0.50$
& $2.277{\times}10^{-5}\pm2.7{\times}10^{-7}$
& $1.701{\times}10^{-1}\pm9.3{\times}10^{-4}$
& $4.758{\times}10^{-1}\pm1.7{\times}10^{-3}$
\\
MSE
& none
& $2.283{\times}10^{-5}\pm2.7{\times}10^{-7}$
& $2.117{\times}10^{-1}\pm1.4{\times}10^{-3}$
& $5.103{\times}10^{-1}\pm1.9{\times}10^{-3}$
\\
\bottomrule
\end{tabular}
}
\end{table}

The clean-MSE results, reported in Table~\ref{tab:replacement-clean-mse}, show the same trend. In the clean realizable case, the quadratic-type losses achieve the smallest clean MSE. However, under replacement corruption, their clean MSE increases rapidly, indicating that the learned alignment is pulled toward mismatched teacher features. In contrast, the proposed SQC-capped loss maintains nearly constant clean MSE across corruption levels, supporting the claim that the cap reduces the influence of plausible but mismatched teacher representations.

\begin{table}[h]
\centering
\caption{
Clean MSE under replacement corruption. The clean MSE is evaluated against the uncorrupted teacher features. Values are mean $\pm$ standard deviation over
20 seeds.
}
\label{tab:replacement-clean-mse}
\resizebox{\linewidth}{!}{
\begin{tabular}{llccc}
\toprule
Loss & Setting & $\rho_{\rm corr}=0$ & $\rho_{\rm corr}=0.2$ & $\rho_{\rm corr}=0.5$\\
\midrule
SQC-capped
& $\tau$-quantile $=0.80,\ \mu_\tau=0.2$
& $9.336{\times}10^{-6}\pm2.8{\times}10^{-6}$
& $\mathbf{7.975{\times}10^{-6}\pm1.2{\times}10^{-6}}$
& $\mathbf{1.048{\times}10^{-5}\pm1.7{\times}10^{-6}}$
\\
Welsch
& scale quantile $=0.50$
& $5.257{\times}10^{-10}\pm4.0{\times}10^{-10}$
& $1.365{\times}10^{-3}\pm3.8{\times}10^{-5}$
& $2.532{\times}10^{-2}\pm4.5{\times}10^{-4}$
\\
Cauchy
& scale quantile $=0.50$
& $\mathbf{4.278{\times}10^{-10}\pm1.1{\times}10^{-11}}$
& $7.417{\times}10^{-3}\pm1.2{\times}10^{-4}$
& $9.875{\times}10^{-2}\pm1.8{\times}10^{-3}$
\\
Pseudo-Huber
& scale quantile $=0.50$
& $4.640{\times}10^{-10}\pm1.6{\times}10^{-11}$
& $1.955{\times}10^{-2}\pm3.0{\times}10^{-4}$
& $1.796{\times}10^{-1}\pm3.2{\times}10^{-3}$
\\
Huber
& scale quantile $=0.50$
& $5.023{\times}10^{-10}\pm1.3{\times}10^{-11}$
& $2.885{\times}10^{-2}\pm4.5{\times}10^{-4}$
& $2.259{\times}10^{-1}\pm3.9{\times}10^{-3}$
\\
MSE
& none
& $5.051{\times}10^{-10}\pm1.3{\times}10^{-11}$
& $4.470{\times}10^{-2}\pm6.8{\times}10^{-4}$
& $2.599{\times}10^{-1}\pm4.5{\times}10^{-3}$
\\
\bottomrule
\end{tabular}
}
\end{table}

Overall, these results suggest that the proposed loss is particularly effective for replacement-type non-Gaussian corruption, where teacher features are plausible but mismatched. This is the regime most relevant to unreliable feature supervision in representation alignment. 


\section{Application to robust feature-alignment distillation}\label{sec:NumRes}
We turn to the optimization behavior of the proposed robust feature-alignment objective. Section~\ref{sec:loss-robustness} studies the effectiveness of the proposed loss relative to several alternative losses by fixing the optimizer and varying the loss function. The experiment in this section fixes the loss and compares different algorithms for
minimizing it. All experiments are implemented in Python
and run on a laptop equipped with a 12th Gen Intel\textsuperscript{\textregistered}
Core\textsuperscript{\texttrademark} i7-12800H CPU at 1.80 GHz and 16 GB of RAM.

All methods are applied to the same corrupted teacher features and the same robust objective
\[
 h_{\mathrm{KD}}(W) =\frac{1}{N}\sum_{i=1}^{N} \Phi_{\tau,\mu}\left(Wz_{s,i}-z_{t,i}^{\mathrm{train}}\right),
\]
where $\Phi_{\tau,\mu}(r)=\min\{\|r\|,\tau\}+\mu \|r\|^2$.
The $\tau$ equals to the $70$-th percentile of the initial residual norms and $\mu=\mu_\tau/\tau$, where $\mu_\tau=0.05$.
For each random seed, the data, the clean teacher features, the corrupted teacher features, the ground-truth matrix $W_\star$, the initial point $W^0$,
and the parameters $\tau$ and $\mu$ are generated once and then passed to all algorithms. Thus, all methods start from exactly the same initial matrix and
optimize exactly the same objective for each seed. In this experiment, $W^0$ is a random matrix generated with the scale $0.2$.

We compare Adam \cite{KingmaBa2015} (Adam-$h_{\mathrm{KD}}$), the subgradient method with a constant step-size \cite{rahimi2024projected,Shor1985} (SubGD-$h_{\mathrm{KD}}$), and the proposed inexact HiPPA with $p=1.5$, $p=2$, and $p=3$.
The methods Adam (with learning rate $\eta = 10^{-3}$) and SubGD (with step-size $\eta = 3\times 10^{-1}$) run for at most $2000$ iterations. The HiPPA methods run for at most $80$ outer iterations, and each proximal subproblem is solved inexactly using $25$ Adam inner iterations together with an acceptance and retry rule based on the decrease of the proximal model $Q_k$.
Denoting $\eta_{\mathrm{in}}$ as the learning rate used by the inner Adam solver for the proximal subproblem, we implemented HiPPA-$p=1.5$ (with $\beta = 8$ and $\eta_{\mathrm{in}} = 3\times 10^{-3}$), HiPPA-$p=2$  (with $\beta = 10$ and $\eta_{\mathrm{in}} = 3\times 10^{-3}$), and
HiPPA-$p=3$: (with $\beta = 30$ and $\eta_{\mathrm{in}} = 3\times 10^{-3}$).

All methods are stopped once the relative recovery error satisfies
$\operatorname{RelErr}:= \frac{\|W-W_\star\|_F}{\|W_\star\|_F} \le 10^{-4}$,
or when the maximum iteration budget is reached. Since the problem is synthetic, we evaluate the methods using the relative
recovery error $ \operatorname{RelErr}$, the clean MSE $ \frac{1}{2N}\sum_{i=1}^{N}\|Wz_{s,i}-z_{t,i}^{\mathrm{clean}}\|^2$,
and the clean robust objective $h_{\mathrm{KD}}(W;Z_t^{\mathrm{clean}})$. We also report the elapsed time.

Table~\ref{tab:algorithm-comparison} reports the results averaged over $20$ independent random seeds. Adam and SubGD do not reach the prescribed relative-error tolerance within the budget of iterations. In contrast, all three HiPPA variants reach the target accuracy for all seeds. Among them, HiPPA with $p=3$ is the fastest. 
\begin{table}[h]
\centering
\caption{
Algorithm comparison for the robust objective $h_{\mathrm{KD}}$ with early
stopping at $\operatorname{RelErr} \le 10^{-4}$. Values are reported
as mean $\pm$ sample standard deviation over $20$ seeds.
}
\label{tab:algorithm-comparison}
\resizebox{\textwidth}{!}{
\begin{tabular}{lcccc}
\toprule
Method
& Rel. $W$-err.
& Clean MSE
& Clean $h_{\mathrm{KD}}$
& Time (s) \\
\midrule
Adam-$h_{\mathrm{KD}}$
& $2.923{\times}10^{-4} \pm 5.8{\times}10^{-5}$
& $8.904{\times}10^{-8} \pm 3.5{\times}10^{-8}$
& $4.113{\times}10^{-4} \pm 7.8{\times}10^{-5}$
& $5.64 \pm 1.08$ \\

SubGD-$h_{\mathrm{KD}}$
& $1.428{\times}10^{-3} \pm 1.2{\times}10^{-5}$
& $2.343{\times}10^{-6} \pm 1.9{\times}10^{-8}$
& $2.123{\times}10^{-3} \pm 7.8{\times}10^{-6}$
& $5.35 \pm 0.87$ \\

HiPPA-$p=1.5$
& $7.387{\times}10^{-5} \pm 7.3{\times}10^{-7}$
& $5.295{\times}10^{-9} \pm 5.6{\times}10^{-11}$
& $1.023{\times}10^{-4} \pm 5.4{\times}10^{-7}$
& $4.09 \pm 0.70$ \\

HiPPA-$p=2$
& $7.201{\times}10^{-5} \pm 5.4{\times}10^{-7}$
& $5.056{\times}10^{-9} \pm 5.1{\times}10^{-11}$
& $9.995{\times}10^{-5} \pm 5.0{\times}10^{-7}$
& $3.19 \pm 0.57$ \\

HiPPA-$p=3$
& $\mathbf{7.191{\times}10^{-5} \pm 6.7{\times}10^{-7}}$
& $\mathbf{5.047{\times}10^{-9} \pm 7.9{\times}10^{-11}}$
& $\mathbf{9.986{\times}10^{-5} \pm 7.8{\times}10^{-7}}$
& $\mathbf{2.48 \pm 0.41}$ \\
\bottomrule
\end{tabular}
}
\end{table}

Figures~\ref{fig:alg-recovery-iter} and~\ref{fig:alg-clean-hkd} show the convergence behavior for the representative seed $123$. 
Figure~\ref{fig:alg-recovery-iter} compares the relative $W$-error and clean MSE as functions of the logged iteration. All methods start from the same
initial value. HiPAA exhibits a delayed but sharp decrease after the initial phase. Among them, HiPPA-$p=3$ reaches the prescribed accuracy in the
fewest logged iterations. Adam decreases faster at the beginning, but then stagnates above the tolerance. SubGD decreases more slowly and stops at a
higher error level.

\begin{figure}[ht]
\centering
\begin{subfigure}{0.48\textwidth}
   \centering
   \includegraphics[width=\textwidth]{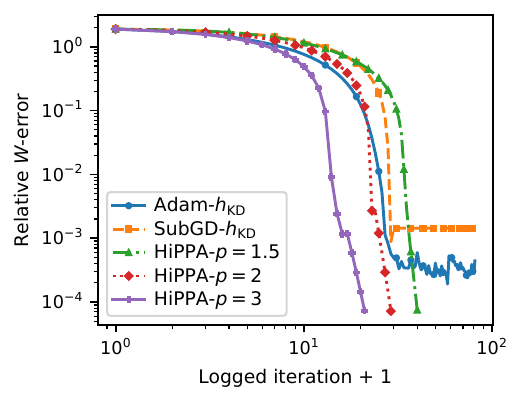}
   \caption{Relative $W$-error}
   \label{fig:alg-rel-w-iter}
\end{subfigure}
\hspace{0.01\textwidth}
\begin{subfigure}{0.48\textwidth}
   \centering
   \includegraphics[width=\textwidth]{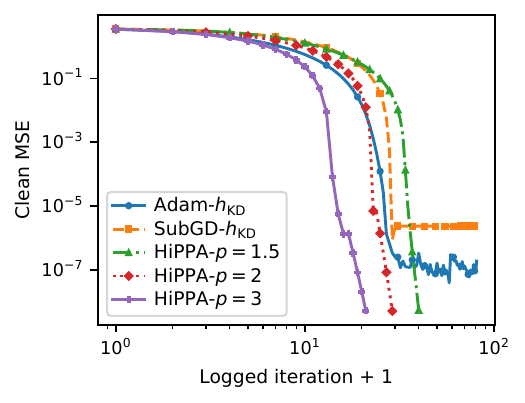}
   \caption{Clean MSE}
   \label{fig:alg-clean-mse-iter}
\end{subfigure}
\caption{Recovery performance on the robust objective $h_{\mathrm{KD}}$ over iterations.}
\label{fig:alg-recovery-iter}
\end{figure}

Figure~\ref{fig:alg-clean-hkd} shows the clean robust objective versus logged iteration and elapsed time. The time plot confirms the computational advantage
of the $p=3$ variant under the stopping rule: HiPPA-$p=3$ reaches the accuracy threshold earlier than the other HiPPA variants and earlier than the
direct methods. Adam and SubGD use the full computational budget without reaching the target relative error, while the HiPPA methods stop earlier after
meeting the prescribed tolerance.

\begin{figure}[ht]
\centering
\begin{subfigure}{0.48\textwidth}
   \centering
   \includegraphics[width=\textwidth]{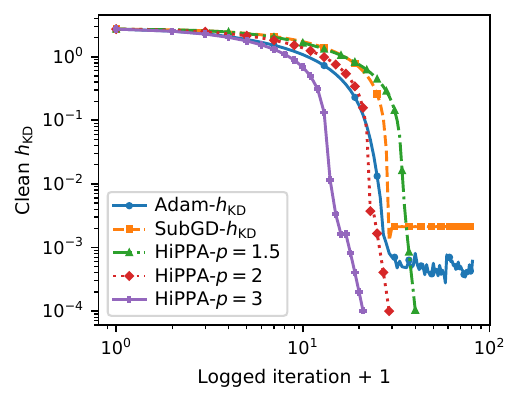}
   \caption{Clean $h_{\mathrm{KD}}$ vs. logged iteration}
   \label{fig:alg-clean-hkd-iter}
\end{subfigure}
\hspace{0.01\textwidth}
\begin{subfigure}{0.48\textwidth}
   \centering
   \includegraphics[width=\textwidth]{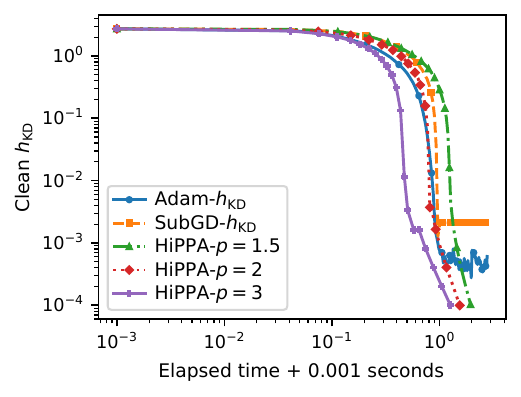}
   \caption{Clean $h_{\mathrm{KD}}$ vs. elapsed time}
   \label{fig:alg-clean-hkd-time}
\end{subfigure}
\caption{Clean robust objective for different algorithms. }
\label{fig:alg-clean-hkd}
\end{figure}

Overall, Adam provides a reasonable baseline but does not reach the target recovery accuracy within the budget, while SubGD is less accurate and fails to meet the stopping tolerance. In contrast, all HiPPA variants achieve the target accuracy across all 20 seeds. Among them, HiPPA-$p=3$ is the most efficient, requiring the fewest gradient evaluations and lowest runtime, highlighting the benefit of higher-order proximal regularization for robust feature-alignment problems.

\section{Discussion}\label{sec:discussion}
We studied efficient proximal-point methods for computing global minimizers of (strongly) quasar-convex objectives in robust learning. We showed that problems arising in robust feature-alignment distillation and anisotropic model stitching are strongly quasar-convex, guaranteeing a unique global minimizer and no saddle points. We then proposed HiPPA, an inexact high-order proximal-point method that controls subproblem accuracy via the model-value gap, and established global convergence with local linear or superlinear rates depending on the regularization order and accuracy criterion. Experiments on robust feature-alignment distillation demonstrate strong empirical performance, highlighting the method’s practical potential.

A key advantage of our high-order proximal-point framework is its flexibility across objective structures, requiring only efficient solution of the associated proximal subproblem. In the nonconvex setting, this subproblem can be solved using methods such as subgradient schemes \cite{davis2018subgradient,rahimi2024projected}, Bregman proximal gradient method \cite{ahookhosh2024high}, or BELLA \cite{ahookhosh2021bregman}. However, it remains unclear which approach is most efficient in practice.

\section*{Acknowledgments}
MA was partially supported by the Research Foundation Flanders (FWO) research project G081222N and UA BOF DocPRO4 projects with ID 46929 and 48996. FL was partially supported by ANID--Chile under project Fondecyt Regular 1241040.

\appendix
\section{Further examples: Strong quasar-convexity beyond local Lipschitzness}
Here, we present a simple mechanism for constructing functions that are nonsmooth, not necessarily locally Lipschitz, nonconvex, and not star-convex, yet strongly quasar-convex on bounded convex sets. The construction is based on a positively homogeneous residual model raised to a fractional power $\theta\in(0,1)$. We first prove a master lemma and then apply it to robust LAD regression, robust multi-task regression, and robust matrix sensing.

\begin{theorem}[\textbf{Master lemma}]\label{thm:master}
Let $K\subset \mb(\ov{x},R)$ be a nonempty convex set with $\ov{x}\in K$, where $R>0$. Let $\theta\in(0,1)$, and let $\psi:K\to[0,\infty)$ satisfy 
\begin{equation}\label{eq:radial-psi}
\psi(\lambda \ov{x}+(1-\lambda)x)=(1-\lambda)\psi(x),\qquad \forall x\in K,\ \forall \lambda\in[0,1],
\end{equation}
and
\begin{equation}\label{eq:psi-lower}
\psi(x)\ge m_\psi\|x-\ov{x}\|,\qquad \forall x\in K,
\end{equation}
for some $m_\psi>0$. Define
\[
h(x):=\psi(x)^\theta.
\]
Then, for every $\kappa\in(0,\theta)$, the function $h$ is $(\kappa,\gamma)$-strongly quasar-convex on $K$ with respect to $\ov{x}$, where
\begin{equation}\label{eq:master-gamma}
\gamma=\frac{2(\theta-\kappa)}{\kappa}\,m_\psi^\theta R^{\,\theta-2}.
\end{equation}
Moreover, if $\psi$ is not identically zero on $K$, then $h$ is not convex and not star-convex with respect to $\ov{x}$. In this case, $h$ is also not locally Lipschitz at $\ov{x}$ relative to $K$.
\end{theorem}
\begin{proof}
Fix $x\in K$ and $\lambda\in[0,1]$. By \eqref{eq:radial-psi},
\[
h(\lambda \ov{x}+(1-\lambda)x)=\bigl((1-\lambda)\psi(x)\bigr)^\theta=(1-\lambda)^\theta h(x).
\]
Since $\theta\in(0,1)$, the map $t\mapsto t^\theta$ is concave on $[0,1]$, and hence $(1-\lambda)^\theta\le 1-\theta\lambda$, i.e.,
\begin{equation}\label{eq1:thm:master}
h(\lambda \ov{x}+(1-\lambda)x)\le(1-\theta\lambda)h(x)=(1-\kappa\lambda)h(x)-(\theta-\kappa)\lambda h(x).
\end{equation}
From \eqref{eq:psi-lower}, we obtain
\[
h(x)=\psi(x)^\theta\ge m_\psi^\theta\|x-\ov{x}\|^\theta.
\]
It follows from $x\in K\subset \mb(\ov{x},R)$ and $\theta-2<0$ that
\[
 \|x-\ov{x}\|^\theta=\|x-\ov{x}\|^2\|x-\ov{x}\|^{\theta-2}\ge R^{\theta-2}\|x-\ov{x}\|^2,
\]
i.e.,
\[
 h(x)\ge m_\psi^\theta R^{\theta-2}\|x-\ov{x}\|^2.
\]
Substituting this estimate into \eqref{eq1:thm:master}, it can be concluded that
\[
h(\lambda \ov{x}+(1-\lambda)x)\le(1-\kappa\lambda)h(x)-(\theta-\kappa)\lambda m_\psi^\theta R^{\theta-2}\|x-\ov{x}\|^2.
\]
Due to $1-\frac{\lambda}{2-\kappa}\le 1$, the choice \eqref{eq:master-gamma} implies
\[
(\theta-\kappa)\lambda m_\psi^\theta R^{\theta-2}\|x-\ov{x}\|^2\ge\lambda\left(1-\frac{\lambda}{2-\kappa}\right)\frac{\kappa\gamma}{2}\|x-\ov{x}\|^2,
\]
i.e.,
\[
h(\lambda \ov{x}+(1-\lambda)x)\le(1-\kappa\lambda)h(x)-\lambda\left(1-\frac{\lambda}{2-\kappa}\right)\frac{\kappa\gamma}{2}\|x-\ov{x}\|^2.
\]
Since $h(\ov{x})=\psi(\ov{x})^\theta=0$ by \eqref{eq:radial-psi} with $\lambda=1$, this is exactly the definition of $(\kappa,\gamma)$-strong quasar-convexity.

Now, suppose that $\psi$ is not identically zero on $K$. Then there exists $d$ such that $\ov x+d\in K$ and $\psi(\ov x+d)>0$. It follows from the convexity of
$K$, $\ov x+td\in K$ for every $t\in[0,1]$, and by \eqref{eq:radial-psi},
\[
        h(\ov x+td)=t^\theta h(\ov x+d).
\]
It can be deduced from $\theta\in(0,1)$ that $t^\theta>t$ for every $t\in(0,1)$, i.e., 
\[
h(\ov{x}+td)>t\,h(\ov{x}+d)=(1-t)h(\ov{x})+t\,h(\ov{x}+d),
\]
due to $h(\ov{x})=0$. Thus, $h$ is not convex. The same strict inequality also shows that the star-convex inequality fails, i.e., $h$ is not star-convex with respect to
$\ov{x}$.

Finally, it holds that
\[
\frac{|h(\ov x+td)-h(\ov x)|}{\|td\|}=\frac{t^\theta h(\ov x+d)}{t\|d\|}=\frac{h(\ov x+d)}{\|d\|}t^{\theta-1}\to+\infty
\]
as $t\downarrow0$. Hence $h$ is not locally Lipschitz at $\ov x$ relative to $K$.
\end{proof}

\begin{remark}
All examples below are local in the radius $R$. Theorem~\ref{thm:master} provides strong quasar-convexity on a bounded convex set
$K\subset\mb(\ov x,R)$, not necessarily on the whole ambient space.
\end{remark}

\subsection{Nonconvex robust LAD regression}

\begin{proposition}\label{prop:lad}
Let $A\in\mathbb{R}^{m\times n}$ have the full column rank, let $b=A\ov{x}$, and let $\theta\in(0,1)$. Define
\[
h_1(x):=\|Ax-b\|_1^\theta,
\]
on a nonempty convex set $K\subset \mb(\ov{x},R)$ with $\ov{x}\in K$. Then, for every $\kappa\in(0,\theta)$, the function $h_1$ is $(\kappa,\gamma_1)$-strongly quasar-convex on $K$ with respect to $\ov{x}$, where
\[
\gamma_1=\frac{2(\theta-\kappa)}{\kappa}\,\sigma_{\min}(A)^\theta\,R^{\theta-2}.
\]
Moreover, $h_1$ is not locally Lipschitz at $\ov x$, is nonconvex, and is
not star-convex with respect to $\ov x$.
\end{proposition}

\begin{proof}
Let us define
\[
\psi_1(x):=\|A(x-\ov{x})\|_1,
\]
i.e.,
\[
\psi_1(\lambda \ov{x}+(1-\lambda)x)=(1-\lambda)\psi_1(x),
\qquad \forall x\in K,\ \forall \lambda\in[0,1].
\]
In addition,
\[
\psi_1(x)\ge \|A(x-\ov{x})\|_2\ge \sigma_{\min}(A)\|x-\ov{x}\|_2.
\]
Consequently, Theorem~\ref{thm:master} applies to $m_\psi=\sigma_{\min}(A)$.

The nonsmoothness at $\ov{x}$ is immediate from $h_1(\ov{x}+td)=t^\theta h_1(\ov{x}+d)$,
for every direction $d$ such that $A d\neq 0$.
\end{proof}

\subsection{Nonconvex robust multi-task regression}

Let $x_i\in\mathbb{R}^d$ and $W^\ast\in\mathbb{R}^{p\times d}$, and let us set $y_i=W^\ast x_i\in\mathbb{R}^p$. Here, we equip $\mathbb{R}^{p\times d}$ with the Frobenius norm $\|\cdot\|_F$.

\begin{proposition}\label{prop:multitask}
Let $X=[x_1,\dots,x_N]\in\mathbb{R}^{d\times N}$ have full row rank, let $\theta\in(0,1)$, and define
\[
h_2(W):=\left(\frac{1}{N}\sum_{i=1}^N \|W x_i-y_i\|_2\right)^\theta,
\]
on a nonempty convex set
\[
K\subset \{W\in\mathbb{R}^{p\times d}:\|W-W^\ast\|_F\le R\}, \qquad W^\ast\in K.
\]
Then, for every $\kappa\in(0,\theta)$, the function $h_2$ is $(\kappa,\gamma_2)$-strongly quasar-convex on $K$ with respect to $W^\ast$, where
\[
\gamma_2=\frac{2(\theta-\kappa)}{\kappa}\left(\frac{\sigma_{\min}(X)}{N}\right)^\theta R^{\theta-2}.
\]
Moreover, $h_2$ is not locally Lipschitz at $W^\star$, is nonconvex, and is not star-convex with respect to $W^\star$.
\end{proposition}
\begin{proof}
Let us define
\[
\psi_2(W):=\frac{1}{N}\sum_{i=1}^N \|(W-W^\ast)x_i\|_2,
\]
i.e.,
\[
\psi_2(\lambda W^\ast+(1-\lambda)W)=(1-\lambda)\psi_2(W),\qquad \forall W\in K,\ \forall \lambda\in[0,1].
\]
In addition,
\[
\psi_2(W)\ge\frac1N\left(\sum_{i=1}^N \|(W-W^\ast)x_i\|_2^2\right)^{1/2}=\frac1N\|(W-W^\ast)X\|_F\ge\frac{\sigma_{\min}(X)}{N}\|W-W^\ast\|_F.
\]
Consequently, Theorem~\ref{thm:master} applies with
\[
m_\psi=\frac{\sigma_{\min}(X)}{N}.
\]
Again, the nonsmoothness at $W^\ast$ follows from
\[
h_2(W^\ast+tD)=t^\theta h_2(W^\ast+D),
\]
for any matrix direction $D$ such that $DX\neq 0$.
\end{proof}

\subsection{Nonconvex robust matrix sensing}

Let $\mathcal{A}:\mathbb{R}^{m\times n}\to\mathbb{R}^p$ be a linear operator, let $b=\mathcal{A}(X^\ast)$, and let us equip $\mathbb{R}^{m\times n}$ with the Frobenius norm.

\begin{proposition}\label{prop:matrixsensing}
Assume that $\mathcal{A}$ is injective, let $\theta\in(0,1)$, and define
\[
h_3(X):=\|\mathcal{A}(X)-b\|_1^\theta,
\]
on a nonempty convex set
\[
K\subset \{X\in\mathbb{R}^{m\times n}:\|X-X^\ast\|_F\le R\},\qquad X^\ast\in K.
\]
Then, for every $\kappa\in(0,\theta)$, the function $h_3$ is $(\kappa,\gamma_3)$-strongly quasar-convex on $K$ with respect to $X^\ast$, where
\[
\gamma_3=\frac{2(\theta-\kappa)}{\kappa}\,m_{\mathcal A}^\theta\,R^{\theta-2},\qquad m_{\mathcal A}:=\min_{\|U\|_F=1}\|\mathcal A(U)\|_1>0.
\]
Moreover, $h_3$ is not locally Lipschitz at $X^*$, is nonconvex, and is not star-convex with respect to $X^*$.
\end{proposition}
\begin{proof}
Let us define
\[
\psi_3(X):=\|\mathcal{A}(X-X^\ast)\|_1,
\]
i.e.,
\[
\psi_3(\lambda X^\ast+(1-\lambda)X)=(1-\lambda)\psi_3(X),\qquad \forall X\in K,\ \forall \lambda\in[0,1].
\]
Since $\mathcal A$ is injective, the continuous map $U\mapsto \|\mathcal A(U)\|_1$ is strictly positive on the compact unit sphere $\{U:\|U\|_F=1\}$, hence
$m_{\mathcal A}>0$ and
\[
\psi_3(X)\ge m_{\mathcal A}\|X-X^\ast\|_F.
\]
The result follows from Theorem~\ref{thm:master} with $m_\psi=m_{\mathcal A}$.
\end{proof}

\bibliographystyle{siamplain}
\bibliography{references.bib}

@article{ahookhosh2024high,
  title={High-order methods beyond the classical complexity bounds: Inexact high-order proximal-point methods},
  author={Ahookhosh, Masoud and Nesterov, Yurii},
  journal={Mathematical Programming},
  volume = {208},
  pages={365--407},
  year={2024},
  publisher={Springer}
}

@article{ahookhosh2025high,
author = {Ahookhosh, Masoud and Nesterov, Yurii},
title = {High-order methods beyond the classical complexity bounds: Inexact high-order proximal-point methods with segment search},
journal = {Submitted manuscript},
year = {2025},
}

@article{ahookhosh2021bregman,
author = {Ahookhosh, Masoud and Themelis, Andreas and Patrinos, Panagiotis},
title = {A {B}regman forward-backward linesearch algorithm for nonconvex composite optimization: Superlinear convergence to nonisolated local minima},
journal = {SIAM Journal on Optimization},
volume = {31},
pages = {653--685},
year = {2021}
}

@article{Ahookhosh2025Asymptotic,
  author = {Ahookhosh, M. and Iusem, A. and Kabgani, A. and Lara, F.},
  title = {Asymptotic convergence analysis of high-order proximal-point methods beyond sublinear rates},
  year = {2025},
  url = {https://doi.org/10.48550/arXiv.2505.20484},
  journal = {arXiv:2505.20484},
  pages = {}
}

@article{Ahookhosh2026Quasar,
  author        = {Ahookhosh, M. and de Brito, J. and Kabgani, A. and Lara, F. and Yuan, J.},
  title         = {Quasar-convex optimization: Fundamental properties and high-order proximal-point methods},
  year          = {2026},
  journal        = {arXiv: 2604.26735},
  pages = {}
}

@article{barre2023principled,
  title={Principled analyses and design of first-order methods with inexact proximal operators},
  author={Barré, Mathieu and Taylor, Alex B. and Bach, Francis},
  journal={Mathematical Programming},
  volume={201},
  pages={185--230},
  year={2023},
  publisher={Springer}
}

@inproceedings{Barron2019AdaptiveRobustLoss,
  author    = {Barron, J. T.},
  title     = {A general and adaptive robust loss function},
  booktitle = {Proceedings of the IEEE/CVF Conference on Computer Vision and Pattern Recognition},
  pages     = {4331--4339},
  year      = {2019}
}

@article{brito2025extending,
  title={Extending linear convergence of the proximal point algorithm: The quasar-convex case},
  author={de Brito, Jos{\'e} and Lara, Felipe and Liu, Di},
  journal={arXiv preprint arXiv:2509.04375},
  year={2025}
}

@article{Brooks2011Ramp,
  author  = {Brooks, J. P.},
  title   = {Support vector machines with the ramp loss and the hard margin loss},
  journal = {Operations Research},
  volume  = {59},
  number  = {2},
  pages   = {467--479},
  year    = {2011}
}

@article{Csiszarik2021RepresentationMatching,
  author        = {Csisz{\'a}rik, A. and K{\H{o}}r{\"o}si-Szab{\'o}, P. and Matszangosz, {\'A}. and Papp, G. and Varga, D.},
  title         = {Similarity and matching of neural network representations},
  year          = {2021},
  journal        = {arXiv: 2110.14633},
  journal = {}
}

@book{Clarke1990,
author = {Clarke, Frank H.},
title = {Optimization and Nonsmooth Analysis},
publisher = {Society for Industrial and Applied Mathematics},
year = {1990},
}

@article{davis2018subgradient,
  author    = {Davis, Damek and Drusvyatskiy, Dmitriy and MacPhee, Kyle J. and Paquette,Courtney},
  title     = {Subgradient Methods for Sharp Weakly Convex Functions},
  journal   = {Journal of Optimization Theory and Applications},
  year      = {2018},
  volume    = {179},
  pages     = {962--982}
}

@article{duchi2019variance,
  author  = {Duchi, J. and Namkoong, H.},
  title   = {Variance-based regularization with convex objectives},
  journal = {Journal of Machine Learning Research},
  volume  = {20},
  number  = {68},
  pages   = {1--55},
  year    = {2019}
}

@article{Dvurechenskii2022,
  author    = {Dvurechensky, P. E.},
  title     = {A gradient method with inexact oracle for composite nonconvex optimization},
  journal   = {Computer Research and Modeling},
  year      = {2022},
  volume    = {14},
  pages     = {321--334}
}

@article{farzin2025minimization,
  title={Minimisation of quasar-convex functions using random zeroth-order oracles},
  author={Farzin, Amir Ali and Pun, Yuen-Man and Braun, Philipp and Shames, Iman},
  journal={arXiv preprint arXiv:2505.02281},
  year={2025}
}

@article{guminov2023accelerated,
  title={Accelerated methods for weakly-quasi-convex optimization problems: S. Guminov et al.},
  author={Guminov, Sergey and Gasnikov, Alexander and Kuruzov, Ilya},
  journal={Computational Management Science},
  volume={20},
  pages={36},
  year={2023},
  publisher={Springer}
}

@article{Hermant2024Study,
  title={Study of the behaviour of {N}esterov accelerated gradient in a non convex setting: the strongly quasar convex case},
  author={Hermant, Julien and Aujol, Jean-Fran{\c{c}}ois and Dossal, Charles and Rondepierre, Aude},
  journal={arXiv preprint arXiv:2405.19809},
  year={2024}
}

@book{hadjisavvas2006handbook,
  title={Handbook of generalized convexity and generalized monotonicity},
  author={Hadjisavvas, Nicolas and Koml{\'o}si, S{\'a}ndor and Schaible, Siegfried S},
  volume={76},
  year={2005},
  publisher={Springer-Verlag}
}

@article{Hardt2018Gradient,
  author  = {Moritz Hardt and Tengyu Ma and Benjamin Recht},
  title   = {Gradient Descent Learns Linear Dynamical Systems},
  journal = {Journal of Machine Learning Research},
  year    = {2018},
  volume  = {19},
  pages   = {1--44}
}

@InProceedings{Hinder2020Near,
  title = 	 {Near-Optimal Methods for Minimizing Star-Convex Functions and Beyond},
  author =       {Hinder, Oliver and Sidford, Aaron and Sohoni, Nimit},
  booktitle = 	 {Proceedings of Thirty Third Conference on Learning Theory},
  pages = 	 {1894--1938},
  year = 	 {2020},
  editor = 	 {Abernethy, Jacob and Agarwal, Shivani},
  volume = 	 {125},
  series = 	 {Proceedings of Machine Learning Research},
  publisher =    {PMLR}
}

@article{Hu2021LoRA,
  author        = {Hu, E. J. and others},
  title         = {{LoRA}: Low-rank adaptation of large language models},
  year          = {2021},
  journal        = {arXiv: 2106.09685},
  pages = {}
}

@article{Huang2014RampLossSVM,
  author  = {Huang, X. and Shi, L. and Suykens, J. A. K.},
  title   = {Ramp loss linear programming support vector machine},
  journal = {Journal of Machine Learning Research},
  volume  = {15},
  pages   = {2185--2211},
  year    = {2014}
}

@article{Huber1964,
  author  = {Huber, P. J.},
  title   = {Robust estimation of a location parameter},
  journal = {The Annals of Mathematical Statistics},
  volume  = {35},
  number  = {1},
  pages   = {73--101},
  year    = {1964}
}

@article{kabgani2024itsopt,
  author = {Kabgani, A. and Ahookhosh, M.},
  title = {Its{OPT}: An inexact two-level smoothing framework for nonconvex optimization via high-order Moreau envelope},
  year = {2024},
  url = {https://doi.org/10.48550/arXiv.2410.19928},
  journal = {arXiv},
  archiveprefix = {2410.19928}
}

@article{kabgani2025itsdeal,
  author = {Kabgani, A. and Ahookhosh, M.},
  title = {Its{DEAL}: Inexact two-level smoothing descent algorithms for weakly convex optimization},
  year = {2025},
  url = {https://doi.org/10.48550/arXiv.2501.02155},
  journal = {arXiv},
  archiveprefix = {2501.02155}
}

@article{kabgani2025fundamental,
  author = {Kabgani, A. and Ahookhosh, M.},
  title = {On fundamental properties of high-order forward-backward envelope},
  year = {2025},
  url = {https://doi.org/10.48550/arXiv.2511.10421},
  journal = {arXiv},
  archiveprefix = {2511.10421}
}

@article{kabgani2025first,
  author = {Kabgani, A. and Ahookhosh, M.},
  title = {First-order majorization-minimization meets high-order majorant: Boosted inexact high-order forward-backward method},
  year = {2025},
  url = {https://doi.org/10.48550/arXiv.2510.22231},
  journal = {arXiv},
  archiveprefix = {2510.22231}
}

@article{kabgani2025moreau,
  author = {Kabgani, A. and Ahookhosh, M.},
  title = {Moreau Envelope and Proximal-Point Methods Under the Lens of High-Order Regularization},
  year = {2025},
  volume={33, 47},
  journal={Set-Valued and Variational Analysis}
}

@inproceedings{KingmaBa2015,
  author    = {Kingma, D. P. and Ba, J.},
  title     = {Adam: A method for stochastic optimization},
  booktitle = {International Conference on Learning Representations (ICLR)},
  year      = {2015}
}

@article{Lara2022strongly,
  title={On strongly quasiconvex functions: Existence results and proximal point algorithms},
  author={Lara, Felipe},
  journal={Journal of Optimization Theory and Applications},
  volume={192},
  number={3},
  pages={891--911},
  year={2022},
  publisher={Springer}
}

@inproceedings{lee2016optimizing,
  title={Optimizing star-convex functions},
  author={Lee, Jasper CH and Valiant, Paul},
  booktitle={2016 IEEE 57th Annual Symposium on Foundations of Computer Science (FOCS)},
  pages={603--614},
  year={2016},
  organization={IEEE}
}

@article{liu2012implementable,
  author  = {Liu, Y.-J. and Sun, D. and Toh, K.-C.},
  title   = {An implementable proximal point algorithmic framework for nuclear norm minimization},
  journal = {Mathematical Programming},
  volume  = {133},
  pages   = {399--436},
  year    = {2012}
}

@article{Maekawa2026Align,
  author        = {Maekawa, S. and Aminnaseri, M. and Pezeshkpour, P. and Hruschka, E.},
  title         = {Align then train: efficient retrieval adapter learning},
  year          = {2026},
  journal        = {arXiv: 2604.03403},
  pages = {}
}

@article{martinet1970regularisation,
  author = {Martinet, Bernard},
  title = {R\'{e}gularisation d'in\'{e}quations variationnelles par approximations successives},
  journal = {Revue Francaise d'informatique et de Recherche operationelle},
  volume = {4},
  pages = {154--158},
  year = {1970}
}

@article{martinet1972determination,
  author = {Martinet, Bernard},
  title = {D\'{e}termination approch\'{e}e d'un point fixe d'une application pseudo-contractante},
  journal = {Cas de l'application prox,”Comptes Rendus de l'Academie des Sciences, Paris},
  volume = {274},
  pages = {163--165},
  year = {1972}
}

@article{nesterov2006cubic,
  title={Cubic regularization of {N}ewton method and its global performance},
  author={Nesterov, Yurii and Polyak, Boris T},
  journal={Mathematical programming},
  volume={108},
  pages={177--205},
  year={2006},
  publisher={Springer}
}

@article{nesterov2023inexact,
  author = {Nesterov, Yurii},
  title = {Inexact Accelerated High-Order Proximal-Point Methods},
  journal = {Mathematical Programming},
  volume = {197},
  pages = {1--26},
  year = {2023}
}

@article{Pun2024,
  title={Online non-stationary stochastic quasar-convex optimization},
  author={Pun, Yuen-Man and Shames, Iman},
  journal={arXiv preprint arXiv:2407.03601},
  year={2024}
}

@article{rahimi2024projected,
  author        = {Rahimi, M. and Ghaderi, S. and Moreau, Y. and Ahookhosh, M.},
  title         = {Projected subgradient methods for paraconvex optimization: Application to robust low-rank matrix recovery},
  year          = {2024},
  journal        = {arXiv:2501.00427},
  pages = {}
}

@article{rockafellar1976monotone,
author = {Rockafellar, R. Tyrrell},
title = {Monotone Operators and the Proximal Point Algorithm},
journal = {SIAM Journal on Control and Optimization},
volume = {14},
pages = {877-898},
year = {1976}
}

@book{Rockafellar09,
  title={Variational Analysis},
  author={Rockafellar, R. Tyrrell and Wets, Roger J-B.},
  year={2009},
  publisher={Springer Berlin, Heidelberg}
}

@inproceedings{Romero2015FitNets,
  author    = {Romero, A. and Ballas, N. and Kahou, S. E. and Chassang, A. and Gatta, C. and Bengio, Y.},
  title     = {FitNets: Hints for thin deep nets},
  booktitle = {International Conference on Learning Representations},
  year      = {2015}
}

@article{Salzo12,
  author = {Salzo, Saverio and Villa, Stefano},
  title = {Inexact and Accelerated Proximal Point Algorithms},
  journal = {Journal of Convex Analysis},
  volume = {19},
  pages = {1167--1192},
  year = {2012}
}

@book{Shor1985,
  author    = {Shor, N. Z.},
  title     = {Minimization Methods for Non-Differentiable Functions},
  publisher = {Springer},
  year      = {1985}
}

@article{Solodov01,
author = {M. V. Solodov and B. F. Svaiter},
title = {A UNIFIED FRAMEWORK FOR SOME INEXACT PROXIMAL POINT ALGORITHMS},
journal = {Numerical Functional Analysis and Optimization},
volume = {22},
pages = {1013--1035},
year = {2001}
}

@article{song2022learning,
  author  = {Song, H. and Kim, M. and Park, D. and Shin, Y. and Lee, J. G.},
  title   = {Learning from noisy labels with deep neural networks: A survey},
  journal = {IEEE Transactions on Neural Networks and Learning Systems},
  volume  = {34},
  number  = {11},
  pages   = {8135--8153},
  year    = {2023}
}

@inproceedings{sra2012scalable,
  author    = {Sra, S.},
  title     = {Scalable nonconvex inexact proximal splitting},
  booktitle = {Advances in Neural Information Processing Systems},
  volume    = {25},
  pages     = {539--547},
  year      = {2012}
}

@article{wang2023continuized,
  title={Continuized acceleration for quasar convex functions in non-convex optimization},
  author={Wang, Jun-Kun and Wibisono, Andre},
  journal={arXiv preprint arXiv:2302.07851},
  year={2023}
}

@article{Wang2019CappedL1,
  author  = {Wang, C. and Ye, Q. and Luo, P. and Ye, N. and Fu, L.},
  title   = {Robust capped {$L_1$}-norm twin support vector machine},
  journal = {Neural Networks},
  volume  = {114},
  pages   = {47--59},
  year    = {2019}
}

@article{Yang2023CappedL1PTSVM,
  author  = {Yang, L. and others},
  title   = {Robust capped {L1}-norm projection twin support vector machine},
  journal = {Journal of Industrial and Management Optimization},
  year    = {2023}
}

\end{document}